\documentclass[11pt]{amsart}
\usepackage{mathptmx,microtype,palatino}
\usepackage[T1]{fontenc}
\usepackage{amsmath, amscd, amsthm} 
\usepackage{amssymb, amsfonts} 
\usepackage{stmaryrd}
\usepackage{enumerate}
\usepackage[svgnames]{xcolor} 
\usepackage{color}
\usepackage[margin=1.5in]{geometry}
\usepackage{mathrsfs}
\usepackage{booktabs}
\usepackage{aliascnt}
\usepackage{mathrsfs}
\usepackage{hyperref}
\usepackage{breqn}
 \definecolor{dark-red}{rgb}{0.4,0.15,0.15}
\hypersetup{
    colorlinks, linkcolor=dark-red,
    citecolor=DarkBlue, urlcolor=MediumBlue
}
\usepackage{mathrsfs}
\usepackage[
textwidth=3cm,
textsize=small,
colorinlistoftodos]
{todonotes}

\usepackage{xargs}

\usepackage{float}
\usepackage{tikz-cd}
\usepackage{subcaption}
\usepackage{mathpartir}
\usepackage{adjustbox}
\usepackage{tikz-3dplot}
\usepackage{nicematrix}
\usepackage{booktabs}
\usepackage{longtable}

\tdplotsetmaincoords{120}{120}

\renewcommand{\emptyset}{\varnothing}

\newcommandx{\sq}[5][5=1]{
  \begin{tikzpicture}[
        baseline=(current bounding box.mid),
        empty/.style={circle, draw=black, inner sep=#5pt},
        dot/.style={empty,fill=black},
        grid/.style={black!70, dotted},
      ]
  \foreach \x in {0,#5}
    \foreach \y in {0,#5}{
       \node[dot] at (\x,\y) {}; 
       \ifthenelse{  \lengthtest{\x pt < #5 pt}  }{
         \draw[grid] (\x,\y) -- (\x+#5,\y);
       }{}
       \ifthenelse{  \lengthtest{\y pt < #5 pt}  }{
         \draw[grid] (\x,\y) -- (\x,\y+#5);
       }{}
  }

  \ifthenelse{ #1 = 1} {
    \draw (0,0) -- (0,#5);
  }{}
  \ifthenelse{ #2 = 1} {
    \draw (0,0) -- (#5,0);
  }{}
  \ifthenelse{ #3 = 1} {
    \draw (0,#5) -- (#5,#5);
  }{}
  \ifthenelse{ #4 = 1} {
    \draw (#5,0) -- (#5,#5);
  }{}
\end{tikzpicture}
}

\newcommandx{\subgrid}[5][5=1]{
  \begin{tikzpicture}[
      baseline=(current bounding box.mid),
      empty/.style={circle, draw=black, inner sep=#5*3pt, outer sep=0pt},
      dot/.style={empty, fill=black}
      ]
      \foreach \x in {0,#5}
        \foreach \y in {0,#5}
        {
          \node[empty] at (\x, \y) {};
        }
  \ifthenelse{ #1 = 1} {
    \node[dot] at (0,0) {};
  }{}
  \ifthenelse{ #2 = 1} {
    \node[dot] at (0,#5) {};
  }{}
  \ifthenelse{ #3 = 1} {
    \node[dot] at (#5,0) {};
  }{}
  \ifthenelse{ #4 = 1} {
    \node[dot] at (#5,#5) {};
  }{}
\end{tikzpicture}
}

\newcommand{\Q}{\mathbb{Q}} 
\newcommand{\QQ}{\Q}

\newcommand{\NN}{\mathbb{N}}

\newcommand{\ZZ}{\mathbb{Z}}

\renewcommand{\setminus}{\smallsetminus}

\newcommand{\CMon}{\mathrm{CMon}}

\newcommand{\Sub}{\operatorname{Sub}}
\newcommand{\SubMon}{\operatorname{SubMon}}
\newcommand{\SubSemi}{\operatorname{SubSemi}}
\newcommand{\SatTr}{\operatorname{SatTr}}

\newcommand{\End}{\operatorname{End}}

\newcommand{\stirling}[2]{\begin{Bmatrix}#1\\#2\end{Bmatrix}}

\newcommand{\chain}[1]{[#1]}

\newcommand{\monop}{\cdot}

\usepackage[T1]{fontenc}

\numberwithin{equation}{section} 

\theoremstyle{definition}

\newtheorem*{theorem*}{Theorem}
\newtheorem*{question}{Question}
\newtheorem*{cor*}{Corollary}

\newaliascnt{theorem}{equation}  
\newtheorem{theorem}[theorem]{Theorem}  
\aliascntresetthe{theorem}

\newaliascnt{dodeca}{equation}  
  
\aliascntresetthe{dodeca}

\newaliascnt{prop}{equation}  
\newtheorem{prop}[prop]{Proposition}
\aliascntresetthe{prop}

\newaliascnt{lemma}{equation}  
\newtheorem{lemma}[lemma]{Lemma}
\aliascntresetthe{lemma}

\newaliascnt{corollary}{equation}  
\newtheorem{corollary}[corollary]{Corollary}
\aliascntresetthe{corollary}

\newaliascnt{claim}{equation}  

\aliascntresetthe{claim}

\newaliascnt{conjecture}{equation}  
\newtheorem{conjecture}[conjecture]{Conjecture}
\aliascntresetthe{conjecture}

\theoremstyle{definition}

\newaliascnt{defn}{equation}  
\newtheorem{defn}[defn]{Definition}
\aliascntresetthe{defn}

\newaliascnt{example}{equation}  
\newtheorem{example}[example]{Example}
\aliascntresetthe{example}

\theoremstyle{remark}

\newaliascnt{remark}{equation}  
\newtheorem{remark}[remark]{Remark}
\aliascntresetthe{remark}

\newaliascnt{convention}{equation}  

\aliascntresetthe{convention}

\theoremstyle{plain}

\newcommand{\aref}[1]{\autoref{#1}}

\begin{document}

\makeatletter
\def\@evenhead{\hfill \small
  \textsc{collaborative mathematics research group}\hfill\scriptsize\thepage}
\makeatother

\title{Enumerating submonoids of finite commutative monoids}

\author[Kirkpatrick]{Caoilainn Kirkpatrick}
\address{Reed College}
\email{k.caoilainn@gmail.com}

\author[el Mahmoud]{Amelie el Mahmoud}
\address{Reed College}
\email{aelmahmoud7@gmail.com}

\author[Ormsby]{Kyle Ormsby}
\address{Reed College}
\email{ormsbyk@reed.edu}

\author[Osorno]{Ang\'elica M. Osorno}
\address{Reed College}
\email{aosorno@reed.edu}

\author[Schandelmeier-Lynch]{Dale Schandelmeier-Lynch}
\address{Reed College}
\email{daleschandelmeierlynch@gmail.com}

\author[Shahar]{Riley Shahar}
\address{University of Pennsylvania}
\email{rshahar@sas.upenn.edu}

\author[Yi]{Lixing Yi}
\address{Reed College}
\email{yilixing233@gmail.com}

\author[Young]{Avery Young}
\address{Reed College}
\email{averycyoung1@gmail.com}

\author[Zhu]{Saron Zhu}
\address{University of Michigan}
\email{zsaron@umich.edu}

\begin{abstract}
Given a finite commutative monoid $M$, we show that submonoids of $M\times [n]$ --- where $[n] = \{0,1,\ldots,n\}$ is equipped with the max operation $\vee$ --- may be enumerated via the transfer matrix method. When $M$ is also idempotent, we show that there are finitely many integers $\lambda$ and rational numbers $b_\lambda$ (only depending on $M$) such that the number of submonoids of $M\times [n]$ is $\sum_\lambda b_\lambda\lambda^n$. This answers a question of Knuth regarding ternary (and higher order) max-closed relations, and has applications to the enumeration of saturated transfer systems in equivariant infinite loop space theory.
\end{abstract}

\maketitle

\setcounter{tocdepth}{1}
\tableofcontents

\section{Introduction}\label{sec:intro}

Subgroup lattices have long played an important role in the theory of finite groups. While $\Sub G$ is far from a complete invariant of a finite group $G$, many important properties of $G$ are faithfully represented by the isomorphism type of $\Sub G$. In light of this, subgroup lattices have enjoyed intense study; see \cite{schmidt} for an overview.

In this paper, we turn our attention to a more general but less ubiquitous algebraic structure: monoids and their submonoid lattices, with particular emphasis on finite commutative and finite commutative idempotent monoids. See \autoref{sec:submonoids} for recollections of essential definitions.

The main problem we attack is easy to state. For a natural number $n$, let $\chain n = \{0,1,\ldots,n\}$ with the operation $k\vee \ell = \max\{k,\ell\}$.

\begin{question}
Given a finite monoid $M$, to what extent does $\SubMon M$ determine the cardinality of $\SubMon(M\times\chain n)$?
\end{question}

Varying the hypotheses on $M$, we arrive at several satisfactory answers. Our most general result employs the \emph{transfer matrix method}, which is nicely summarized in \cite[Chapter 4]{EC}. The idea is that, when counted with appropriate weights, length $n$ walks in a particular weighted directed graph $G(M)$ count the submonoids of $M\times \chain n$. These weighted walks are computed by powers of the adjacency matrix $W=W(M)$ of $G(M)$, thus allowing us to bring linear algebraic tools to bear on the problem.

\begin{theorem*}[\autoref{thm:walks}]
Suppose $M$ is a finite commutative monoid. Then there is a lower triangular matrix $W$ of natural numbers with rows and columns indexed by a linear extension of $\SubMon M$ such that
\[
  \#\SubMon(M\times \chain n) = \sum_{A,B\in \SubMon M} (W^n)_{A,B}.
\]
\end{theorem*}

The entries of $W$ are combinatorial in nature, with
\[
  W_{A,B} = \#\{I\subseteq A\mid I\text{ an ideal of $A$ and }I\cup B = A\}.
\]
(See \autoref{defn:ideal} and \autoref{defn:submon_graph}.) Carefully analyzing these numbers for $M$ a finite Abelian group results in a surprising corollary:

\begin{theorem*}[\autoref{thm:group}]
Suppose $M$ is a finite Abelian group and $n$ is a natural number. Then
\[
  \#\SubMon(M\times \chain n) = 2^n\cdot\sum_{m\ge 0}c_m2^{-m}\binom{n}{m} \sim \frac{1}{h!} c_h n^h 2^{n-h},
\]
where $c_m$ is the number of length $m$ chains in the subgroup lattice of $M$ and $h$ is the \emph{height} of $\Sub M$, \emph{i.e.}, the length of the longest chain in the subgroup lattice.
\end{theorem*}

Specializing in another direction, we consider the case of finite idempotent monoids. These are equivalent to finite join-semilattices; see \autoref{ex:monoids}(b).

\begin{theorem*}[\autoref{thm:diag}, \autoref{cor:fmla}]
Suppose $M$ is a finite commutative idempotent monoid. Then the matrix $W(M)$ is diagonalizable. It follows that there is a finite set of positive integers $\Lambda = \Lambda(M)$ and rational numbers $b_\lambda = b_\lambda(M)$ for $\lambda\in \Lambda$ such that
\[
  \#\SubMon(M\times \chain n) = \sum_{\lambda\in \Lambda}b_\lambda \lambda^n.
\]
\end{theorem*}

The above theorem provides a partial answer to a question of Don Knuth \cite[Problem 10]{knuth:parades}. Knuth asks how many ``max-closed relations'' there are on a finite product of chains $\prod \chain{n_i}$. His max-closed relations are just subsemigroups of $\prod \chain{n_i}$ under join. It is straightforward to check that for a join-semilattice $P$, the ``remove $e$'' map is two-to-one from subsemigroups to submonoids, and hence
\[
  \# \SubSemi \prod \chain{n_i} = 2\#\SubMon \prod \chain{n_i}.
\]
Knuth \cite{knuth:parades} details how $\SubSemi(\chain m\times \chain n)$ is enumerated by the \emph{poly-Bernoulli numbers} defined by Kaneko \cite{kaneko}; see \autoref{subsec:pB}. \autoref{cor:fmla} provides an inductive method for attacking this problem, but remains too unwieldy for Knuth's stated goal of determining $\#\SubSemi(\chain{8}^3)$. Indeed, we would first need to diagonalize the matrix $W(\chain{8}^2)$ which is $k\times k$ with
\[
  k = \frac 12B_{9,9} = 22\,111\,390\,122\,811;
\]
given that storing such a matrix would require around 2 million zettabytes, this is simply too large for practical computation.

The connection between commutative idempotent monoids and join-semilattices provides additional motivation for our work. Indeed, we undertook this project not with monoids in mind, but rather with the goal of understanding saturated transfer systems on lattices. These are combinatorial structures on subgroup lattices that are important in equivariant infinite loop space theory; see \autoref{sec:matchstick} and \cite{hmoo,Rubin}. In this context, $M=\Sub G$ (for $G$ a finite Abelian group) equipped with the join operation. It is a consequence of \cite{mrc-ormsby} that saturated transfer systems on $\Sub G$ are in bijective correspondence with $\SubMon(\Sub G)$.

For $G = C_N$ cyclic of order $N = p_1^{n_1}\cdots p_r^{n_r}$ with $p_i$ distinct primes and $n_i$ natural numbers,
\[
  \Sub C_N \cong \chain{n_1}\times\cdots\times \chain{n_r}.
\]
As such, the problem of enumerating saturated transfer systems on $\Sub C_N$ is equivalent to Knuth's problem on max-closed relations on $\prod \chain{n_i}$, and \autoref{cor:fmla} again provides traction.

\begin{remark}
Let $\CMon$ denote the category of commutative monoids and monoid homomorphisms (see \autoref{sec:submonoids} for definitions), and write $\CMon/\chain n$ for the slice category of commutative monoids equipped with a homomorphism to $\chain n$. We can view objects $p\colon N\to \chain n$ of $\CMon/\chain n$ as strongly multiplicative length $n+1$ filtrations by ideals. Here an ideal $I$ is a subset $I\subseteq N$ satisfying $IN=I$, and the filtration by ideals takes the form
\[
  I_n\subseteq I_{n-1}\subseteq \cdots \subseteq I_{1}\subseteq I_0=N
\]
where $I_k = p^{-1}[k,n]$. 
In order to reconstruct a monoid homomorphism $p\colon N\to\chain n$ from such a chain of ideals, it must also satisfy the \emph{strongly multiplicative} criterion
\[
  I_kI_\ell\subseteq I_{k\vee \ell}\setminus I_{(k\vee \ell)+1}
\]
where $k\vee \ell = \max\{k,\ell\}$.

Observe that $M\times\chain n$ is naturally an object of $\CMon/\chain n$ via projection. In fact, the assignment $M\mapsto M\times\chain n$ is left adjoint to the forgetful functor $\CMon/\chain n\to \CMon$, and thus we may view $M\times\chain n$ as the free commutative monoid equipped with a strongly multiplicative length $n+1$ filtration by ideals. This provides a categorical justification for studying $M\times\chain n$ and its submonoids beyond our applications to max-closed relations and saturated transfer systems.
\end{remark}

\subsection*{Overview}
We now briefly outline the rest of the paper. In \autoref{sec:submonoids}, we review some basics on monoids and submonoids, introduce the (weighted, directed) submonoid graph of a finite commutative monoid $M$, and use this to prove \autoref{thm:walks}, our most general theorem enumerating submonoids of $M\times \chain n$. We specialize this theorem to $M$ a commutative Abelian group in \autoref{thm:group}, and conclude the section with some observations on generating functions.

In \autoref{sec:semilattices}, we suppose that $M$ is finite, commutative, and idempotent, and use these assumptions to derive a particularly special form for $\#\SubMon(M\times\chain n)$ in this case; see \autoref{cor:fmla}. We use this to demonstrate that $\#\SubMon(M\times\chain n)$ is a constant-recursive sequence in this setting, and then make the constants in \autoref{cor:fmla} explicit for a selection of finite commutative idempotent monoids. We conclude \autoref{sec:semilattices} by making the connection between our work, \cite{knuth:parades}, and poly-Bernoulli numbers explicit, including a conjecture regarding multivariable exponential generating functions.

In \autoref{sec:matchstick}, we demonstrate how our techniques apply to the enumeration of saturated transfer systems on finite lattices. This makes our enumerations applicable in the theory of $N_\infty$-operads and equivariant ring spectra \cite{BlumbergHill,Rubin}.

\subsection*{Acknowledgments}
This work was completed by the summer 2024 edition of Reed College's Collaborative Mathematics Research Group, supported by NSF award DMS--2204365.

\section{Transition matrices for submonoids}\label{sec:submonoids}
Throughout this section, let $M = (M,\monop)$ be a finite commutative monoid with identity element $e$; in other words, $M$ is a finite set equipped with binary operation $\monop$ satisfying $x\monop y=y\monop x$, $x\monop (y\monop z) = (x\monop y)\monop z$, and $x\monop e=x = e\monop x$ for all $x,y,z\in M$. A homomorphism between monoids $M,N$ is a function $f\colon M\to N$ satisfying $f(x\monop y) = f(x)\monop f(y)$ and $f(e)=e$. We will often implicitly work in the category of monoids and monoid homomorphisms, adding adjectives like finite, commutative, or idempotent as necessary.

\begin{example}\label{ex:monoids}
We briefly recall some important examples and fix notation:
\begin{enumerate}[(a)]
    \item Any commutative group is a commutative monoid.
    \item Suppose $(P,\le)$ is a finite \emph{join-semilattice}, that is, a partially ordered set admitting least upper bounds (joins) $\vee$. Then $(P,\vee)$ is a commutative \emph{idempotent} ($x\vee x = x$) monoid. These structures are in fact equivalent: if $(M,\monop)$ is a commutative idempotent monoid, then $x\le y\iff x\monop y=y$ makes $(M,\le)$ a join-semilattice.
    \item As an important subexample of (b), for a natural number $n$, the finite chain $[n] = \{0<1<\cdots<n\}$ is a commutative (idempotent) monoid under $\ell\vee m=\max\{\ell,m\}$.
    \item Given two commutative monoids $M,N$, their Cartesian product $M\times N$ is a commutative monoid under $(x,y)\monop (x',y') = (x\monop x',y\monop y')$.
    \item For a set $S$, the power set of $S$ is a commutative (idempotent) monoid with operation given by set union. It is isomorphic to $[1]^{\#S}$.
\end{enumerate}
\end{example}

\begin{defn}\label{defn:submonoid}
    Let $(M,\monop)$ be a monoid. A \emph{submonoid} $A$ of $M$ is a subset of $M$ that contains $e$ and is closed under the operation $\monop$. We denote the collection of submonoids of $M$ by 
    $\SubMon M$ and write $A\le M$ when $A$ is a submonoid of $M$.
\end{defn}

\subsection{The submonoid graph}
We now construct a weighted directed graph $G(M)$ associated with a commutative monoid $M$. Its vertices will consist of submonoids of $M$, and we will ultimately show that the total weight of length $n$ walks in $G(M)$ counts the number of submonoids of $M\times [n]$.

\begin{defn}\label{defn:ideal}
    An \emph{ideal} of a commutative monoid $M$ is a subset $I\subseteq M$ with the property that $I\monop M = I$.
\end{defn}

\begin{remark}
    If $M$ is also idempotent, then ideals of $M$ are the same as upwards closed subsets (\textit{upsets}) of $(M,\le)$.
\end{remark}

\begin{defn}\label{defn:submon_graph}
Suppose $M$ is a finite commutative monoid. The \emph{submonoid graph} of $M$ is the weighted directed graph $G(M)$ with vertices the submonoids of $M$ and weight function $w\colon \SubMon(M) \times \SubMon(M)\to \NN$ given by
\[
  w(A,B) := \#\{I\subseteq A\mid I\text{ an ideal of $A$ and }I\cup B = A\}.
\]
After (sometimes implicitly) choosing a linear ordering of $\SubMon M$, we write $W = W(M)$ for the associated adjacency matrix in which
\[
  W_{A,B} = w(A,B).
\]
\end{defn}

\begin{example}
  Let $M = \chain{1}\times\chain{1}$. Let $A = \{(0, 0), (0, 1), (1, 1)\}$ and
  $B = \{(0, 0), (1, 1)\}$. We draw $B$ in black and $A\setminus B$ in red:
  
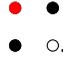
\[
  \begin{tikzpicture}[
      empty/.style={circle, draw=black, inner sep=1.5pt, outer sep=0pt},
      dot/.style={empty, fill=black}
      ]
      \foreach \x in {0,.5}
        \foreach \y in {0,.5}
        {
          \node[empty] at (\x, \y) {};
        }

        \node[dot] at (0, 0) {};
        \node[dot] at (.5, .5) {};
        \node[dot, draw=red, fill=red] at (0, .5) {};
    \end{tikzpicture}.
  \]
  In order to count towards $w(M,N)$, an ideal $I$ must contain all of $A\setminus B = \{(0, 1)\}$, and because of the ideal condition, it must contain $(1,1)$ as well, so our only choice is to either include $(0, 0)$ or not. Thus $w(M,N) =
  2$. The entire submonoid graph of $M$ and its adjacency matrix are displayed in \autoref{fig:submon_graph}. 
\end{example}
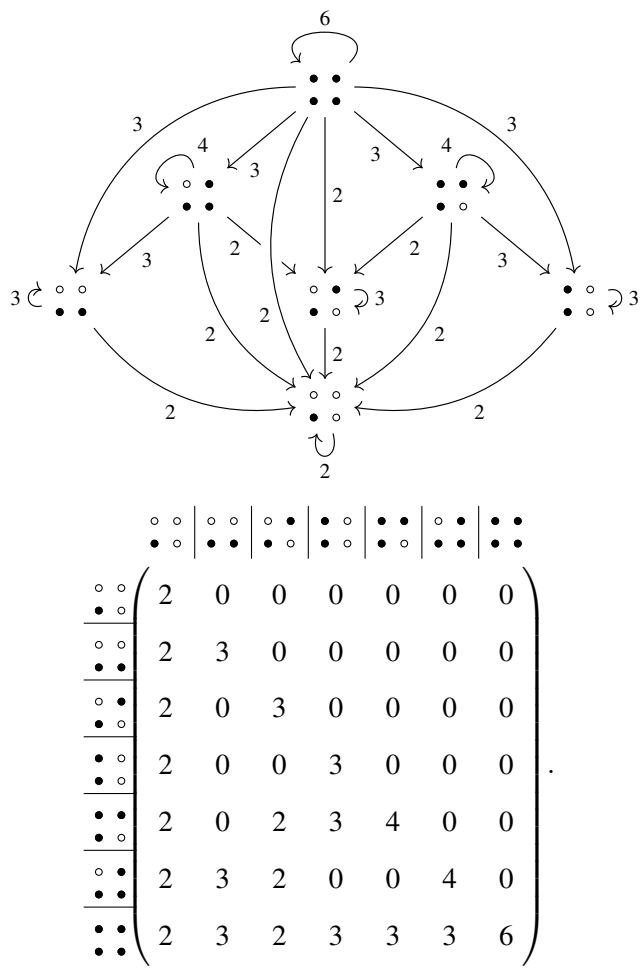
\begin{figure}
  \begin{tikzcd}
    &&\subgrid{1}{1}{1}{1}[.3]
    \ar[loop, looseness=3, "6"']
    \ar[dl, "3",pos=.78]
    \ar[dr, "3"']
    \ar[ddll, bend right=40, "3"']
    \ar[ddrr, bend left=40, "3"]
    \ar[dd, "2"]
    \\
    &\subgrid{1}{0}{1}{1}[.3]
    \ar[loop, out=100, in=170, looseness=3, "4"', pos=.1]
    \ar[dr, "2"', pos=.3]
    \ar[dl, "3"]
    \ar[ddr, bend right, "2"']
    &&\subgrid{1}{1}{0}{1}[.3]
    \ar[loop, out=80, in=10, looseness=3, "4", pos=.1]
    \ar[dl, "2", pos=.3]
    \ar[dr, "3"']
    \ar[ddl, bend left, "2"]
    \\
    \subgrid{1}{0}{1}{0}[.3]
    \ar[loop left, looseness=3, "3"]
    \ar[drr, bend right, "2"']
       &&
    \subgrid{1}{0}{0}{1}[.3]
    \ar[loop right, looseness=3, "3"]
    \ar[d, "2"]
       &&
    \subgrid{1}{1}{0}{0}[.3]
    \ar[loop right, looseness=3, "3"]
    \ar[dll, bend left, "2"]
       \\ &&
    \subgrid{1}{0}{0}{0}[.3]
      \ar[loop below, looseness=5, "2"] 
    \ar[from=1-3, bend right, crossing over, "2"', pos=.7]
  \end{tikzcd}
\[
    \NiceMatrixOptions{cell-space-limits = 7pt}
  \begin{pNiceArray}{ccccccc}[first-row,first-col,name=M]
& \subgrid{1}{0}{0}{0}[.3] & \subgrid{1}{0}{1}{0}[.3] & \subgrid{1}{0}{0}{1}[.3]
& \subgrid{1}{1}{0}{0}[.3] & \subgrid{1}{1}{0}{1}[.3] & \subgrid{1}{0}{1}{1}[.3]
& \subgrid{1}{1}{1}{1}[.3] \\
\subgrid{1}{0}{0}{0}[.3] & 2 & 0 & 0 & 0 & 0 & 0 & 0 \\ 
\subgrid{1}{0}{1}{0}[.3] & 2 & 3 & 0 & 0 & 0 & 0 & 0 \\ 
\subgrid{1}{0}{0}{1}[.3] & 2 & 0 & 3 & 0 & 0 & 0 & 0 \\ 
\subgrid{1}{1}{0}{0}[.3] & 2 & 0 & 0 & 3 & 0 & 0 & 0 \\ 
\subgrid{1}{1}{0}{1}[.3] & 2 & 0 & 2 & 3 & 4 & 0 & 0 \\ 
\subgrid{1}{0}{1}{1}[.3] & 2 & 3 & 2 & 0 & 0 & 4 & 0 \\ 
\subgrid{1}{1}{1}{1}[.3] & 2 & 3 & 2 & 3 & 3 & 3 & 6
 
  \end{pNiceArray}.
\]
\begin{tikzpicture}[remember picture,overlay]
  \def\tick{10pt} 
  \def\dx{5pt}   
  \def\dy{-3.1pt}

  \foreach \c in {1,...,6}{
    \draw[line width=0.4pt]
      ([xshift=\dx,yshift=-\tick]M-0-\c.east) --
      ([xshift=\dx,yshift= \tick]M-0-\c.east);
  }

  \foreach \r in {1,...,6}{
    \draw[line width=0.4pt]
      ([yshift=\dy,xshift=-\tick]M-\r-0.south) --
      ([yshift=\dy,xshift= \tick]M-\r-0.south);
  }


\end{tikzpicture}

\caption{The submonoid graph $G(\chain 1\times \chain 1)$ and the associated adjancency matrix $W(\chain 1\times\chain 1)$. It is an instructive exercise to check that the indicated weights are correct; the omission of an edge indicates weight $0$.}\label{fig:submon_graph}
\end{figure}

We will show that length $n$ walks in $G(M)$ enumerate submonoids of $M\times \chain n$. Perhaps surprisingly, they do so in terms of the projection of a submonoid onto $M\times \chain 0\cong M$. Given $n\in \NN$, let $\pi = \pi_n\colon M\times\chain n\to M$ denote the projection map $\pi(x,k) = x$. The map $\pi$ is a homomorphism of monoids and thus takes submonoids to submonoids.

\begin{lemma}\label{lem:upclosed-recurrence}
Let $M$ be a finite commutative monoid. For $A$ a submonoid of $M$, let
\[
  m_A(n) := \#\{R\le M\times\chain n\mid \pi_nR=A\}.
\]
Then
\[
  m_A(0) = 1\qquad\text{and}\qquad
  m_A(n+1) = \sum_{B\leq M} w(A,B)m_{B}(n).
\]
\end{lemma}
\begin{proof}
The proof proceeds in the following steps:
\begin{enumerate}[1.]
\item Check the base case for $M\times\chain 0\cong M$.
\item From each submonoid $R\le M\times\chain{n+1}$ such that $\pi_{n+1}R=A$, construct a submonoid $R'$ of $M\times \chain n$ together with an ideal $I$ of $A$ such that $I\cup \pi_n R' = A$.
\item Construct an inverse assignment taking the data of $R'\le M\times\chain n$ and an ideal $I$ of $A$ such that $I\cup \pi_n R' = A$ to a submonoid of $M\times\chain{n+1}$ that projects to $A$.
\item Observe that by partitioning on $B = \pi(R')$, the bijection from the previous two steps gives the inductive formula.
\end{enumerate}

Step 1 is immediate; Step 4 follows from Steps 2 and 3.

For Step 2, let $R\le M\times\chain{n+1}$ be a submonoid such that $\pi_{n+1}R=A$. Define $R' = R\cap (M\times \chain n)$ and $I = \pi_{n+1}(R\cap (M\times\{n+1\}))$. It is straightforward to check that $I$ is an ideal of $A$ such that $I\cup \pi_n R' = A$.

For Step 3, let $R'$ be a submonoid of $M\times \chain n$ and $I\subseteq A$ be an ideal of $A$ with $I\cup \pi_n R' = A$. We set $R := R'\cup (I\times \{n+1\})$. It is straightforward to check that $R$ is a submonoid of $M\times\chain{n+1}$ with $\pi_{n+1}R = A$.

Additionally, the assignments of Steps 2 and 3 are inverse by construction, concluding the proof.
\end{proof}

The previous lemma establishes the following theorem.

\begin{theorem}\label{thm:walks}
Let $M$ be a finite commutative monoid, suppose $A,B\in \SubMon(M)$, and let $W = W(M)$ be the adjacency matrix of the submonoid graph $G(M)$ with respect to some linear ordering of $\SubMon(M)$. Then the number of submonoids $R$ of $M\times\chain n$ with $\pi_n R = A$ and $\pi_{n-1}(R\cap (M\times\chain{n-1})) = B$ is the $(A,B)$-entry of $W^n$, and the total number of submonoids of $M\times\chain n$ is
\[
  \#\SubMon(M\times\chain n) = \sum_{A,B\le P}(W^n)_{A,B}.
\]\hfill\qedsymbol
\end{theorem}

In the language of \cite[Chapter 4]{EC}, the above theorem establishes that $W$ is the \emph{transfer matrix} for submonoids of $M\times\chain n$. We can use it to establish asymptotics for $\#\SubMon(M\times \chain n)$ in terms of the number of ideals of submonoids of $M$. To do so, we will need to make some observations regarding the shape of $W(M)$.

To fix terminology, recall that for a partially ordered set $P$, a \emph{linear extension} of $P$ is a total ordering of $P$ that respects the partial order on $P$.

\begin{lemma}\label{lemma:upper_triangular}
Suppose $M$ is a finite commutative monoid and fix a linear extension of $\SubMon(M)$ under inclusion. When written with respect to this ordering, $W(M)$ is lower triangular.
\end{lemma}
\begin{proof}
This is simply the observation that $B \not\subseteq A$ implies that $w(A,B) = 0$.
\end{proof}

To state the following theorem on the asymptotic growth of $\#\SubMon(M\times\chain n)$, we establish the following notation: let $i(M)$ denote the maximal number of ideals of a submonoid of $M$; let $r(M)$ denote the number of submonoids $A\le M$ for which $w(A,A) = i(M)$. (We warn the reader that at this level of generality, $i(M)$ may not be equal to $w(M,M)$, but it is equal to the maximal value of $w(A,A)$ for $A\le M$. When $M$ is idempotent, \autoref{lemma:order_rev} implies that $i(M)=w(M,M)$ and $r(M)=1$.)

\begin{theorem}\label{thm:asymptotics}
If $M$ is a finite commutative monoid, then asymptotically in $n$,
\[
  \#\SubMon(M\times \chain n) = \Theta(n^{k}i(M)^n).
\] 
for some $0\le k\le r(M)-1$. If $r(M)=1$, then
\[
  \#\SubMon(M\times \chain n) \sim C\cdot i(M)^n
\]
for $C=C(M)$ a positive constant.
\end{theorem}
\begin{proof}
We leave the details to the reader, but note that this follows from a standard analysis of the Jordan decomposition of the lower triangular matrix $W(M)$ (\autoref{lemma:upper_triangular}) with nonnegative entries; see for instance~\cite[Appendix D.2]{leveque2007}. The crucial observations are that $i(M)$ is the largest eigenvalue of $W(M)$ and $k$ is the length of the longest path in the induced subgraph (not including self-loops) on $A\le M$ such that $w(A,A)=i(M)$.
\end{proof}

We conclude this section by providing an explicit formula for $\#\SubMon(G\times [n])$ when $G$ is a finite Abelian group.

\begin{theorem}\label{thm:group}
Suppose $M=G$ is a finite Abelian group and $n$ is a natural number. Then
\[
  \#\SubMon(G\times \chain n) = 2^n\cdot\sum_{m\ge 0}c_m2^{-m}\binom{n}{m} \sim \frac{1}{h!} c_h n^h 2^{n-h},
\]
where $c_m$ is the number of length $m$ chains in the subgroup lattice of $G$ and $h$ is the \emph{height} of $\Sub(G)$, \emph{i.e.}, the length of the longest chain in the subgroup lattice.
\end{theorem}
\begin{proof}
First note that the asymptotic formula follows immediately from the explicit formula.

We now determine the entries of $W(G)$. Observe that every submonoid of $G$ is in fact a subgroup of $G$ (since $G$ is a \textit{finite} group). Further note that every finite Abelian group $A$ has exactly two ideals, itself and the empty set. Based on these facts, we can deduce that for $A,B\in \SubMon(G)$,
\[
  w(A,B) =
  \begin{cases}
      2 &\text{if }A=B,\\
      1 &\text{if }B<A,\\
      0 &\text{otherwise}.
  \end{cases}
\]

Let $k = \#\Sub(G)$ be the number of subgroups of $G$. Then $W = W(G)$ is a $k\times k$ lower triangular matrix with $2$s on the diagonal and lower triangle entries either $0$ or $1$ with $W_{A,B}=1$ if and only if $B\subsetneq A$. This allows us to write
\[
  W = 2I_k+L
\]
for $L$ a strictly-lower triangular matrix of $0$s and $1$s with $L_{A,B}=1$ if and only if $B \subsetneq A$.

By the binomial theorem, which can be used since $2I_k$ and $L$ commute, 
\[
  W^n = \sum_{m\ge 0}\binom nm2^{n-m}L^m = 2^n\cdot\sum_{m\ge 0}2^{-m}\binom nm L^m.
\]
Let $\mathbf 1$ be the $k\times 1$ column vector of all $1$s. Then by \autoref{thm:walks},
\[
  \#\SubMon(G\times\chain n) = \mathbf 1^\top W^n\mathbf 1 = 2^n\cdot\sum_{m\ge 0}2^{-m}\binom nm \mathbf 1^\top L^m\mathbf 1.
\]
Now observe that $L$ is the adjacency matrix for the directed graph on $\Sub(G)$ with a directed edge from $A$ to $B$ if and only if $B$ is a proper subgroup of $A$. By the usual interpretation of $L^m$ in terms of length $m$ walks, we see that $\mathbf 1^\top L^m\mathbf 1 = c_m$ the number of length $m$ chains in $\Sub(G)$. This demonstrates that
\[
  \#\SubMon(G\times\chain n) = 2^n\cdot\sum_{m\ge 0}c_m2^{-m}\binom nm
\]
as desired.
\end{proof}

\subsection{Ordinary generating function for $\SubMon(M\times\chain n)$}
Suppose $M$ is a finite commutative monoid and let $W = W(M)$. We now use the results of \cite[Chapter 4]{EC} to deduce some results about generating functions.

Let $F_{ij}(x)$ denote the ordinary generating of the $(i,j)$-entry of $W^n$:
\[
  F_{ij}(x) := \sum_{n\ge 0}(W^n)_{ij}x^n.
\]
By \cite[Theorem 4.7.2]{EC}, we have
\[
  F_{ij}(x) = \frac{(-1)^{i+j}\det(I-xW : j,i)}{\det(I-xW)}
\]
where $(B:j,i)$ denotes the matrix obtained by removing the $j$-th row and $i$-th column of $B$.

Now let $F(x)$ be the ordinary generating function for $\#\SubMon(M\times\chain n)$. By \autoref{thm:walks},
\[
  F(x) = \sum_{i,j}F_{i,j}(x)
\]
so we get the following result.
\begin{prop}\label{prop:ord_gen}
Suppose $M$ is a finite commutative monoid and
\[
  F(x) = \sum_{n\ge 0}\#\SubMon(M\times\chain n)x^n.
\]
Then
\[
  F(x) = \frac{\sum_{i,j}(-1)^{i+j}\det(I-xW:j,i)}{\prod_{A\le M}(1-w(A,A) x)},
\]
where $W = W(M)$ and $w(A,A)$ is the number of ideals of the submonoid $A\le M$.
\end{prop}
\begin{proof}
From the preceding discussion, the only thing left to check is the expression for the denominator of $F(x)$. This follows because $W$ is lower triangular with eigenvalues $w(A,A)$, $A\le M$.
\end{proof}

\section{Enumerating submonoids of join-semilattices}\label{sec:semilattices}
\subsection{General results}
We now specialize to the case where $M = P$ is a finite join-semilattice (\emph{i.e.}, commutative idempotent monoid; see \autoref{ex:monoids}(b)). In \autoref{thm:diag}, we prove that $W = W(P)$ is diagonalizable, and this will give us excellent control over $\#\SubMon(P\times\chain n)$.

Fix a linear extension of $\SubMon P$ so that $W(P)$ is lower triangular (see \autoref{lemma:upper_triangular}). The following lemma is a basic fact from linear algebra. We first encountered it at \cite{mse-diagonalizability}. 

\begin{lemma}\label{lemma:diag}
Suppose $W$ is a triangular square matrix (lower or upper) with entries in a field $F$, and suppose that the diagonal entries of $W$ are arranged so that if two agree, then all diagonal entries between them agree. Then $W$ is diagonalizable over $F$ if and only if each square diagonal block with a single value on its diagonal is in fact a diagonal matrix as depicted in \autoref{fig:matrix}.\hfill\qedsymbol
\end{lemma} 

\begin{figure}
    \centering
    \begin{tikzpicture}[scale=1]

\draw[thick] (0,0) rectangle (6,6);

\draw[thick] (2,1) -- (2,6);
\draw[thick] (5,0) -- (5,4);
\draw[thick] (2,1) -- (6,1);
\draw[thick] (0,4) -- (5,4);

\fill[gray] (0,0) rectangle (2,4);
\fill[gray] (2,0) rectangle (5,1);

\node at (0.5,5.5) {$\lambda_1$};
\node at (1.5,4.5) {$\lambda_1$};
\node at (2.5,3.5) {$\lambda_2$};
\node at (3.5,2.5) {$\lambda_2$};
\node at (4.5,1.5) {$\lambda_2$};
\node at (5.5,0.5) {$\lambda_3$};

\end{tikzpicture}
    \caption{A graphical representation of a $6\times 6$ lower triangular matrix of the shape described in \autoref{lemma:diag}. All blank regions are filled with $0$'s, while the entries in the gray region are arbitrary. A lower triangular matrix with equal diagonal entries contiguous is diagonalizable if and only if it has this shape.}
    \label{fig:matrix}
\end{figure}
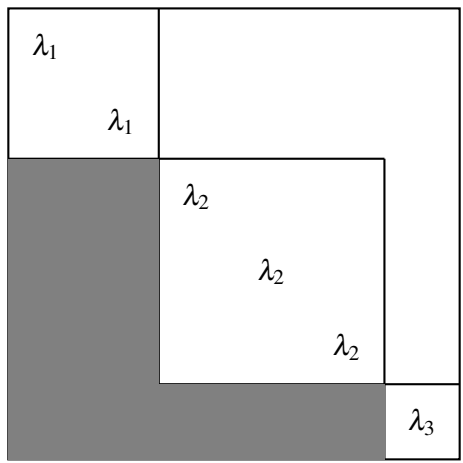

To prove the diagonalizability of $W(P)$, we will show that if two distinct submonoids $A,B\le P$ satisfy $w(A,A)=w(B,B)$, then $A$ and $B$ are incomparable, whence $w(A,B) = 0$. For the purposes of this proof, recall that an \emph{antichain} in a poset is a collection of mutually incomparable elements; it is a basic fact from order theory that antichains are in bijection with ideals. (The bijection takes an antichain to the ideal generated by it, which is given by its upward closure, and it takes an ideal to its set of minimal elements.)

\begin{lemma}\label{lemma:order_rev}
Suppose $P$ is a finite join-semilattice. Then the map
\[
\begin{aligned}
  \SubMon P&\longrightarrow \NN\\
  A&\longmapsto w(A,A)
\end{aligned}
\]
is strictly increasing.
\end{lemma}
\begin{proof}
Suppose $A<B\le P$. We have
\[
  w(A,A) = \#\{\text{ideals of }A\} = \#\{\text{antichains in }A\}.
\]
Antichains in $A$ are also antichains in $B$, so $w(A,A)\le w(B,B)$. Since $A<B$, we may take $x\in B\setminus A$ to exhibit an antichain $\{x\}$ in $B$ that is not an antichain in $A$. Thus $w(A,A)<w(B,B)$, as required.
\end{proof}

Before moving on to diagonalizability, we use \autoref{lemma:order_rev} to pin down the asymptotics of $\#\SubMon(P\times\chain n)$:

\begin{theorem}\label{thm:semilattice_asymptotics}
Suppose $P$ is a finite join-semilattice. Then
\[
  \#\SubMon(P\times \chain n)\sim C\cdot i(P)^n
\]
where $C=C(P)$ is a positive constant and $i(P)=w(P,P)$ is the number of antichains of $P$.
\end{theorem}
\begin{proof}
By \autoref{thm:asymptotics}, it suffices to show that the multiplicity of the eigenvalue $i(P)=w(P,P)$ is one. All eigenvalues are of the form $w(A,A)$, and \autoref{lemma:order_rev} implies that $A\mapsto w(A,A)$ achieves its maximum value $w(P,P)$ once (because $P$ is the unique maximum of $\SubMon P$).
\end{proof}

\begin{theorem}\label{thm:diag}
Let $P$ be a finite join-semilattice. Then $W(P)$ is diagonalizable with eignevalues the number of antichains in $A$, $A\in \SubMon P$.
\end{theorem}
\begin{proof}
This follows directly from \autoref{lemma:diag} and \autoref{lemma:order_rev}.
\end{proof}

The diagonalizability of $W(P)$ has immediate enumerative consequences:

\begin{corollary}\label{cor:fmla}
Suppose $P$ is a finite join-semilattice and let $\Lambda = \Lambda(P)$ be the finite set of positive integers
\[
  \Lambda := \{w(A,A)\mid A\in \SubMon P\}.
\]
For $\lambda\in\Lambda$, let $b_\lambda = b_\lambda(P)$ be the rational number \[
  b_\lambda := \mathbf{1}^TQ_\lambda\mathbf{1},
\] where $Q_\lambda = Q_\lambda(P)$ is the matrix acting by projection onto the $\lambda$-eigenspace of $W(P)$ when considering the decomposition into eigenspaces. 
Then
\[
  \#\SubMon(P\times\chain n) = \sum_{\lambda\in \Lambda}b_\lambda \lambda^n.
\]
\end{corollary}
\begin{proof}
This follows from \autoref{thm:walks} and \autoref{thm:diag} via the usual formul{\ae} for powers of diagonalizable matrices. The formula for the $b_\lambda$ follows from the spectral decomposition of $W(P)$.
\end{proof}

\begin{remark}\label{rmk:interpretation-of-coeffs}
While the natural numbers in $\Lambda(P)$ encode combinatorial structure as the number of antichains in submonoids of $P$, the rational numbers $b_\lambda$ have resisted combinatorial interpretation thus far. The subsequent formul{\ae} we will derive for specific examples demonstrate that these numbers tend to have large denominators, and we view it as an important open question to compute them --- or, equivalently, an eigenbasis of $W(P)$ --- in terms of $P$.  Some explicit computations are included in \autoref{sec:table}.

The decomposition \[
  Q_\lambda = \prod_{\substack{\mu\in\Lambda\\ \mu\neq \lambda}} \frac{W - \mu I}{\lambda - \mu}
\] for the spectral projectors suggests instead attempting to interpret the normalized coefficients $b_\lambda\prod_{\mu\neq\lambda} (\lambda - \mu)$, which are always integers. Despite numerous searches on the OEIS, these values remain similarly mysterious. That said, see \autoref{prop:bvals} for a method to determine the $b_\lambda$ in terms of initial values and a Vandermonde matrix.
\end{remark}

\autoref{cor:fmla} implies that $\#\SubMon(P\times\chain n)$ is a constant-recursive sequence. The following result is standard for such sequences; see \cite[Theorem 4.1.1]{EC}.
\begin{prop}\label{prop:recur}
Let $P$ be a finite join-semilattice, and let $a_i$ be the coefficient of the degree $i$ term in $\prod_{\lambda\in\Lambda}(1-\lambda x)$. Then for $n\ge \#\Lambda$,
\[
  \#\SubMon(P\times\chain n) + \sum_{i=1}^{\#\Lambda} a_i\#\SubMon(P\times\chain{n-i}) = 0.
\]
\end{prop}

\begin{proof}
As before, write $F(x)$ for the ordinary generating function of $\#\SubMon(P\times \chain n)$. By \autoref{cor:fmla}, we can also express $F(x)$ as
\[
  F(x) = \sum_{n\ge 0}\sum_{\lambda\in \Lambda}b_\lambda \lambda^nx^n = \sum_{\lambda\in \Lambda} \frac{b_\lambda}{1-\lambda x} = \frac{f(x)}{\prod_{\lambda\in \Lambda} (1-\lambda x)}
\]
where $f(x)$ is a polynomial with rational coefficients of degree $\#\Lambda-1$. Thus
\[
  \prod_{\lambda\in \Lambda}(1-\lambda x)\sum_{n\ge 0}\#\SubMon(P\times\chain n)x^n = f(x)
\]
and matching the terms of degree $n\ge \#\Lambda$ on both sides gives the recurrence.
\end{proof}

Finally, we observe that the $b_\lambda(P)$ coefficients of \autoref{cor:fmla} can be derived from $\Lambda(P)$ and the initial $\#\Lambda(P)$ terms of $\#\SubMon(P\times\chain n)$.

\begin{prop}\label{prop:bvals}
Suppose $P$ is a finite join-semilattice, let $\Lambda=\Lambda(P) = \{\lambda_1,\ldots,\lambda_k\}$ with cardinality $k$, and let $S_n = \#\SubMon(P\times\chain n)$. Let $V$ be the $k\times k$ Vandermonde matrix
\[
  V = (\lambda_j^r)_{0\le r\le k-1,1\le j\le k}.
\]
Then $b_i := b_{\lambda_i}(P)$ are the unique rational numbers satisfying
\[
  V
  \begin{pmatrix}
      b_1\\ b_2\\ \vdots\\ b_k
  \end{pmatrix}
  =
  \begin{pmatrix}
      S_0\\ S_1\\ \vdots\\ S_{k-1}
  \end{pmatrix}.
\]
\end{prop}
\begin{proof}
By \autoref{cor:fmla}, the exponential generation function $S(z)$ of $S_n$ takes the form
\[
  S(z) := \sum_{n\ge 0} S_n\frac{z^n}{n!} = \sum_{i=1}^k b_ie^{\lambda_i z}.
\]
Thus
\[
  S^{(r)}(0) = \sum_{i=1}^kb_i\lambda_i^r = S_r.
\]
The system of equations given by the right-hand equality for $0\le r\le k-1$ is exactly $Vb=S$ for $b = (b_1,\ldots,b_k)^\top$, $S = (S_0,\ldots,S_{k-1})^\top$. This system has a unique solution because $V$ is invertible. (Indeed, it is a standard fact that $\det V = \prod_{1\le i<j\le k}(\lambda_j-\lambda_i)$, which is nonzero since the $\lambda_i$ are distinct.)
\end{proof}

\subsection{Formul{\ae} for specific semilattices}
For a general finite join-semilattice $P$, the problem of determining $\Lambda(P)$ and the associated $b_\lambda$ seems to be quite hard. Given the relation between $\Lambda(P)$ and the enumeration of antichains, this should come as no surprise. (See \autoref{rmk:oeis} for additional comments on the inherent difficulty of this problem.)

Our first result in this direction is a complete determination of $\Lambda(\chain m)$.
\begin{prop}\label{prop:mLambda}
For a nonnegative integer $m$, the eigenvalues of $W(\chain m)$ are
\[
  \Lambda(\chain m) = \{2,3,\ldots,m+2\}
\]
\end{prop}
\begin{proof}
The submonoids of $[m]$ are exactly the subsets $S\subseteq \chain m$ that contain $0$. Furthermore, $w(S,S) = \#S+1$ since an antichain in a linearly ordered set is either a singleton or empty.
\end{proof}





We now consider the join-semilattice $\chain m\times \chain 1$.
\begin{prop}
For $m$ a nonnegative integer, the eigenvalues of $W(\chain m\times \chain 1)$ are
\[
  \Lambda(\chain m\times\chain 1) = \left\{\lambda\in \ZZ~\middle|~ 2\le \lambda\le \binom{m+3}{2}\text{ and }\lambda\ne \binom{m+3}{2}-1\right\}.
\]
\end{prop}
\begin{proof}[Proof sketch]
We leave a number of details to the reader, but the idea is as follows.

First note it is always the case that the elements of $\Lambda(P)$ are bounded below by $2$ and above by $w(P,P)$, which is equal to $\binom{m+3}{2}$ in this case.

The reader may check that the cardinality function is a grading on the poset of submonoids. Thus to show that $\binom{m+3}{2}-1$ is not in $\Lambda(\chain m\times\chain 1)$, it suffices to show that each submonoid of cardinality $2m+1$ has at most $\binom{m+3}{2}-2$ antichains.  Such submonoids do not contain either $(0,1)$ or $(i,0)$ for some $i>0$. In the former case, the submonoid does not have antichains $\{(0,1)\}$ and $\{(0,1),(1,0)\}$; in the latter, it does not have antichains $\{(i,0)\}$ and $\{(0,1),(i,0)\}$. We conclude that $\binom{m+3}{2}-1$ is not an eigenvalue.

Finally, the reader may perform a case-by-case check that the other potential eigenvalues $\lambda$, $2\le \lambda\le \binom{m+3}{2}-2$ are obtained.
\end{proof}

\begin{remark}
We have not found a general formula for $b_\lambda(\chain m\times\chain 1)$, but we produce some specific values in \autoref{sec:table}.
\end{remark}

Fix $k\ge 1$. We now consider the lattice $M_k = \{\bot,1,2,\ldots,k,\top\}$ in which $\bot$ is the minimum element, $\top$ is the maximum, and $1,2,\ldots,k$ are incomparable to each other. This lattice is the $k$-fold parallel composition of $[2]$ with itself.

\begin{prop}
For $k$ a positive integer, the eigenvalues of $W(M_k)$ are
\[
  \Lambda(M_k) = \{2\}\cup\{2^i+2\mid 0\le i\le k\}.
\]
\end{prop}

\begin{proof}[Proof sketch]
The submonoids of $M_k$ are of the following forms:
\begin{itemize}
\item Type A submonoids are of the form $A_S = S\cup \{\bot,\top\}$ for $S\subseteq\{1,2,\ldots,k\}$.
\item Type B submonoids are of the form $B_i = \{\bot,i\}$ for $1\le i\le k$.
\item The only Type C submonoid is the trivial submonoid $C$.
\end{itemize}
One then checks the following computations by inspection:
\[
\begin{aligned}
    w(A_S,A_S) &= 2^{\#S}+2,\\
    w(A_S,A_T) &= 2^{\#T}+1\text{ for }T\subsetneq S,\\
    w(A_S,B_i) &= 3\text{ for }i\in S,\\
    w(A_S,C) &= 2,\\
    w(B_i,B_i) &= 3,\\
    w(B_i,C) &= 2,\\
    w(C,C) &= 2.
\end{aligned}
\]
Furthermore, for any other choices of $A,B\le M_k$, $w(A,B)=0$. This amounts to a complete computation of $W(M)$ and in particular gives the desired diagonal entries.
\end{proof}

Note that \autoref{fig:submon_graph} gives the weighted graph and matrix in the case $k=2$.

In order to explore the formula of \autoref{cor:fmla} we produced a short Sage Math script that computes $W(P)$, its diagonalization, and an eigenbasis. We have recorded some results in \autoref{sec:table}.

\subsection{Max-closed relations and poly-Bernoulli numbers}\label{subsec:pB}
In \cite{knuth:parades}, Don Knuth recounts the combinatorics of poly-Bernoulli numbers in terms of parades, acyclic bipartite orientations, doubly bounded permutations, lonesum matrices, and other structures. One may view the \textit{max-closed relations} he considers in \cite[\S14]{knuth:parades} as subsemigroups of $[m-1]\times [n-1]$. Since there are twice as many subsemigroups as submonoids (with two-to-one mapping given by $(-)\cup \bot$), we deduce that
\[
  \#\SubMon([m-1]\times[n-1]) = \frac{1}{2}B_{m,n},
\]
where $B_{m,n}$ is the $(m,n)$-th \textit{poly-Bernoulli number} \cite{kaneko}, with $B_{m,n} = B_n^{(-m)}$ and
\[
  \sum_{n\ge 0}B_n^{(s)}\frac{z^n}{n!} = \frac{1}{1-e^{-z}}\sum_{k\ge 1}\frac{(1-e^{-z})^k}{k^s}.
\]

\begin{theorem}[{\cite{kaneko,knuth:parades}}]
The poly-Bernoulli numbers may be expressed as
\[
  B_{m,n} = \sum_{k\ge 0}(-1)^{n+k}k!\stirling nk (k+1)^m = \sum_{k\ge 0}k!^2\stirling{m+1}{k+1}\stirling{n+1}{k+1}.
\]
Their exponential generating function is given by
\[
  \sum_{m,n\ge 0}B_{m,n}\frac{w^mz^n}{m!n!} = \frac{e^{w+z}}{e^w+e^z-e^{w+z}}.
\]
\hfill\qedsymbol
\end{theorem}

From this, we deduce the following in the notation of \autoref{cor:fmla}.

\begin{theorem}\label{thm:m-by-n}
For $m\ge 0$ and $2\le j\le m+2$,
\[
  b_j(\chain m) = \frac{1}{2}(-1)^{m+j}j!\stirling{m+1}{j-1}.
\]
\hfill\qedsymbol
\end{theorem}

\begin{remark}
The above formula was independently deduced in \cite[Theorem 4.7]{hmoo} using the language of saturated transfer systems (see \autoref{sec:matchstick}).
\end{remark}

We now give an independent computation of these coefficients using the technology of \autoref{cor:fmla}. The important point is that, once the matrix $W$ is identified, the problem is completely reduced to computations of linear algebra.

\begin{proof}[Proof of \autoref{thm:m-by-n}]
  The submonoids $A\leq [m]$ are exactly the subsets of $[m]$ containing $0$; the non-empty ideals of $A$ are in bijection with the elements of $A$ itself, via sending each ideal to its least element. Let $A, B \leq [m]$. By inspection we thus see that \[
    w(A,B) = \begin{cases}
      \#A + 1 & A = B \\
      \#\{a\in A: a\leq\min(A\setminus B)\} & B\subsetneq A \\
      0 & B\not\subseteq A.
    \end{cases}
  \]
  These entries define the matrix $W = W([m])$.

  We claim that the matrix \[
    Q_{A,B}([m]) = \begin{cases}
      (-1)^{\#A - \#B} \displaystyle\prod_{a\in A\setminus B} (1 + \#\{b\in B: b\leq a\}) & B\subseteq A \\
      0 & B\not\subseteq A
    \end{cases}
  \]
  is an eigenmatrix for $W$, satisfying that $Q^{-1}WQ$ is diagonal. We prove this by showing by induction on $\#A-\#B$ that $(WQ)_{A,B} = (\#B + 1)Q_{A,B}$ for all $B \subseteq A$. The base case corresponds to $B=A$ and follows directly from the definitions.
  
  Given $B\subsetneq A$, let $a_0 = \min(A\setminus B)$, and let $r = 1 + \#\{b\in B: b\leq a_0\}$. Then for $B\subseteq C\subseteq A\setminus \{a_0\}$, we have \[
    W_{A,C} = W_{A,B}=r,\quad W_{A, C\sqcup\{a_0\}} = W_{A\setminus \{a_0\}, C} + 1,\quad Q_{C\sqcup\{a_0\},B} = -rQ_{C,B}.
  \]
  Splitting the sum $(WQ)_{A,B} = \sum_{B\subseteq C\subseteq A} W_{A,C}Q_{C,B}$ into the cases $a_0\in C$ and $a_0\notin C$, applying the above computations, and using the inductive hypothesis gives the inductive step. Thus $Q$ is an eigenmatrix.

  We now wish to compute $Q^{-1}\mathbf{1}$ and $\mathbf{1}^TQ$. We first claim that $(Q^{-1}\mathbf{1})_A = \frac{1}{2}(\#A + 1)!$.
  This follows from the fact that $\sum_{B\subseteq A} \frac{1}{2}(\#B + 1)!Q_{A,B} = 1$. The latter is proved by induction on $\#A$, by splitting the sum according to whether $B$ contains $\max(A)$ or not.
  
  Let $r_i$ denote the size of the ``$i$-th run'' in $[m]\setminus B$. Explicitly, this means that, if $B = \{0 = b_1 < \dots < b_{\#B}\}$ and $b_{\# B + 1} = m + 1$, we let $r_i = b_{i + 1} - b_i - 1$. Then $r_i$ is the number of choices of $x\in [m]\setminus B$ so that $1 + \#\{b\in B: b\leq x\} = i$. Thus \[
    (\mathbf{1}^TQ)_B = \sum_{A\supseteq B} Q_{A,B} = (-1)^{m + 1 - \#B}\sum_{x\in [m]\setminus B} (1 + \#\{b\in B: b\leq x\}) = (-1)^{m + 1 - \#B} \prod_{i=1}^{\# B} i^{r_i}.
  \]

  Now, \[
    b_{j}([m]) = \sum_{\#A = j - 1} (\mathbf{1}^TQ)_A(Q^{-1}\mathbf{1})_A = \frac{j!}{2}(-1)^{m - j} \sum_{r_1 + \dots + r_j = m - j} \prod_{i=1}^{j - 1} i^{r_i}.
  \] We conclude \autoref{thm:m-by-n} thanks to the standard identity \[
    \sum_{r_1 + \dots + r_j = m - j} \prod_{i=1}^{j - 1} i^{r_i} = \stirling{m + 1}{j - 1};
  \] see for instance \cite[Equation 26.8.5]{nist-handbook}.
\end{proof}

We now consider the multivariable exponential generating function for the number of submonoids of $\chain \ell\times\chain m\times\chain n$. We define
\[
  G(x,y,z) := \sum_{\ell,m,n\ge 0}\#\SubMon(\chain{\ell-1}\times\chain{m-1}\times\chain{n-1})\frac{x^\ell y^mz^n}{\ell!m!n!}
\]
and also set
\[
  F(t,u,v) = G(x,y,z)
\]
where $t = e^x,u=e^y,v=e^z$. Because $\#\SubMon(\chain{\ell-1}\times\chain{m-1})$ is equal to the poly-Bernoulli number $B_{\ell,m}$, we have
\[
  F(t,u,1) = tu\qquad\text{and}\qquad \left.\frac{\partial F}{\partial v}\right|_{(t,u,1)} = \frac{tu}{t+u-tu}.
\]
By \autoref{cor:fmla}, we also have
\[
  \left.\frac{\partial G}{\partial^\ell x \partial^m y}\right|_{(0,0,z)} = \sum_{\lambda\in \Lambda(\chain{\ell -1}\times\chain{m-1})}b_\lambda e^{\lambda z}.
\]
This partial information is consistent with the following conjecture.

\begin{conjecture}
The generating function $G(x,y,z)$ is a rational function in $e^x,e^y,e^z$, as is the higher order multivariable version for $\#\SubMon(\prod_i \chain{n_i-1})$.
\end{conjecture}

An explicit formula for $F(t,u,v)$ (and hence, for $G(x,y,z)$) and its higher order variants would give a very satisfying answer to the most general version of \cite[Problem 10]{knuth:parades}. 

\begin{remark}\label{rmk:oeis}
The transfer matrix method seems to fall short in enumerating submonoids of $\prod \chain{n_i}$ because it does not sufficiently account for symmetry. We also note that \cite[Remark 2.10]{echt-2023} establishes that submonoids of $\chain{1}^n$ are counted by OEIS entry A102896, which is only known for $n\le 7$, thus suggesting the difficulty inherent in establishing base cases in an inductive approach to enumerating $\SubMon \prod \chain{n_i}$.
\end{remark}

\section{Connection to equivariant stable homotopy theory}\label{sec:matchstick}

In this section we elaborate on the connection between submonoids of a semilattice and saturated transfer systems, which are important objects of study in equivariant stable homotopy theory, as they are closely related to equivariant operads. For more context, we refer the reader to \cite{hmoo, Rubin, notices_transfer}.

\begin{defn}
    Let $(P,\leq )$ be a finite lattice. A \emph{saturated transfer system} on $P$ is a partial order $R$ on $P$ satisfying the following conditions:
    \begin{enumerate}\label{defn:sattransfer}
        \item $R$ refines $\leq$, \emph{i.e.}, if $xRy$, then $x\leq y$;
        \item $R$ is closed under restriction, \emph{i.e.}, if $xRz$ and $y\in P$, then $(x \wedge y)R(z \wedge y)$;
        \item $R$ is decomposable, \emph{i.e.}, if $xRz$ and $x\leq y\leq z$, then $xRy$ and $xRz$.
    \end{enumerate} 
    We denote by $\SatTr(P)$ the set of saturated transfer systems of $P$. This is a poset by refinement (\emph{i.e.}, inclusion).
\end{defn}

Condition (3) is called the saturation condition, a partial order that satisfies (1) and (2) is called a transfer system.

We can consider a saturated transfer system in terms of its Hasse diagram, and as such, we can consider the connected components. Proposition 2.1 of \cite{echt-2023} shows that each connected component has a unique minimal element. The following result combines \cite[Theorem 2.8]{echt-2023} and \cite[Theorem 4.5]{mrc-ormsby}. 

\begin{theorem}\label{thm:bij-sattr-subm}
    There is an order reversing bijection 
    \[\chi\colon \SatTr(P) \to \SubMon(P,\vee)\]
    given by taking a saturated transfer system $R$ to the set minimal elements of the connected components with respect to $R$.
\end{theorem}

In fact, \cite{echt-2023} shows an order reversing bijection between $\SatTr(P)$ and $\End^\circ(P)$, where the latter denotes the set of interior operators of $P$, that is order preserving maps that are idempotent (for all $x\in P$, $f(f(x))=f(x)$) and contractive (for all $x\in P$, $f(x)\leq x$). This is a poset using the pointwise partial order, where $f\leq g$ is and only if $f(x)\leq g(x)$ for all $x\in P$. The result in \cite{mrc-ormsby} connects $\End^\circ(P)$ with $\SubMon(P,\vee)$.

Via this bijection, we can translate each of the enumerative and asymptotic results of \autoref{sec:semilattices} to give analogous results for the count of saturated transfer systems on $P\times [n]$. In the rest of this section, we show how to obtain the transition matrix directly from the data of saturated transfer systems. We present independent proofs of the enumeration techniques, as these ideas might be of interest to those studying transfer systems.

\begin{remark}
The authors actually discovered the results for transfer systems on finite lattices first, and only later realized that they could be generalized to submonoids of finite commutative monoids. We chose to emphasize submonoids in our treatment since they are both more general and more familiar.
\end{remark}

\subsection{Saturated transfer systems on $P\times [n]$}

Let $P$ be a finite lattice. We are interested in saturated transfer
systems $R$ on $P\times \chain{n}$. It will be convenient to have some language
to discuss a given system of saturated covers on this lattice. The below definition is illustrated in \autoref{fig:layer}.

\begin{defn}\label{defn:layer}
  Given a saturated transfer system $S$ on $P\times \chain{n}$, for $i=0,\dots,n$, we define the \emph{$i$-th layer}, denoted by $S_i$, as the restriction of $S$ to $P\times \{i\}$. We also write $S_i$ for the corresponding partial order on $P$ obtained via the isomorphism $P\times \{i\}\cong P$. Note that $S_i$ is indeed a saturated transfer system on $P$. For $i=1,\dots,n$, we write $S^i$ for the restriction of $S$ to the height one cylinder $P\times \{i-1<i\}$. Note that since $P\times \{i-1<i\}$ is isomorphic to $P\times [1]$, we can think of $S^i$ as a saturated transfer system on $P\times [1]$, with bottom and top layers given by $S_{i-1}$ and $S_i$, respectively.
\end{defn}

\tikzset{every picture/.style={line width=0.75pt}} 

\begin{figure}[h]
\begin{tikzpicture}[x=0.75pt,y=0.75pt,yscale=-1,xscale=1]

\draw[dotted]    (202.82,208.5) -- (205.36,242.33) ;
\draw[dotted]     (221.41,226.24) -- (223.11,329.36) ;
\draw[dotted]     (254.36,261.73) -- (254.79,293.03) ;
\draw[dotted]     (287.74,260.88) -- (289.86,327.67) ;
\draw[dotted]     (271.26,207.65) -- (272.96,310.77) ;
\draw[dotted]     (237.46,206.81) -- (239.16,309.93) ;
\draw[dotted]     (287.74,226.24) -- (220.99,227.93) ;
\draw[dotted]     (271.26,207.65) -- (204.51,209.34) ;
\draw[dotted]     (272.96,310.77) -- (206.21,312.46) ;
\draw[dotted]     (289.86,258.38) -- (223.11,260.07) ;
\draw[dotted]     (271.27,275.28) -- (204.52,276.97) ;
\draw[dotted]     (272.11,240.64) -- (205.36,242.33) ;
\draw[dotted]     (223.11,329.36) -- (206.21,312.46) ;
\draw[dotted]     (220.99,227.93) -- (204.09,211.03) ;
\draw[dotted]     (288.16,224.55) -- (271.26,207.65) ;
\draw[dotted]     (254.78,225.4) -- (237.89,208.5) ;
\draw[dotted]     (289.86,327.67) -- (272.96,310.77) ;
\draw[dotted]     (254.79,293.93) -- (237.89,277.03) ;
\draw[dotted]     (256.48,259.23) -- (239.58,242.33) ;
\draw [color={rgb, 255:red, 208; green, 2; blue, 27 }  ,draw opacity=1 ]   (287.74,226.24) -- (287.74,260.88) ;
\draw [color={rgb, 255:red, 208; green, 2; blue, 27 }  ,draw opacity=1 ]   (289.86,292.18) -- (289.86,326.83) ;
\draw [color={rgb, 255:red, 208; green, 2; blue, 27 }  ,draw opacity=1 ]   (205.36,242.33) -- (206.21,312.46) ;
\draw [color={rgb, 255:red, 208; green, 2; blue, 27 }  ,draw opacity=1 ]   (206.21,312.46) -- (213.79,320.05) -- (223.11,329.36) ;
\draw [color={rgb, 255:red, 208; green, 2; blue, 27 }  ,draw opacity=1 ]   (289.86,326.83) -- (223.11,329.36) ;
\draw [color={rgb, 255:red, 208; green, 2; blue, 27 }  ,draw opacity=1 ]   (254.79,292.61) -- (256.06,330.21) ;
\draw [color={rgb, 255:red, 208; green, 2; blue, 27 }  ,draw opacity=1 ]   (254.36,227.09) -- (254.36,261.73) ;
\draw [color={rgb, 255:red, 208; green, 2; blue, 27 }  ,draw opacity=1 ]   (220.99,227.93) -- (223.11,329.36) ;
\draw [color={rgb, 255:red, 208; green, 2; blue, 27 }  ,draw opacity=1 ]   (204.52,276.97) -- (221.42,293.87) ;
\draw [color={rgb, 255:red, 208; green, 2; blue, 27 }  ,draw opacity=1 ]   (205.36,242.33) -- (222.26,259.23) ;
\draw [color={rgb, 255:red, 208; green, 2; blue, 27 }  ,draw opacity=1 ]   (239.58,311.62) -- (206.21,312.46) ;
\draw [color={rgb, 255:red, 208; green, 2; blue, 27 }  ,draw opacity=1 ]   (240.01,309.93) -- (256.48,328.52) ;
\draw [color={rgb, 255:red, 208; green, 2; blue, 27 }  ,draw opacity=1 ]   (288.17,291.34) -- (221.42,293.87) ;
\draw [color={rgb, 255:red, 208; green, 2; blue, 27 }  ,draw opacity=1 ]   (289.86,292.18) -- (223.11,329.36) ;
\draw [color={rgb, 255:red, 208; green, 2; blue, 27 }  ,draw opacity=1 ]   (205.79,277.4) -- (223.11,329.36) ;
\draw [color={rgb, 255:red, 208; green, 2; blue, 27 }  ,draw opacity=1 ]   (205.36,242.33) -- (223.11,329.36) ;
\draw [color={rgb, 255:red, 208; green, 2; blue, 27 }  ,draw opacity=1 ]   (239.58,311.19) -- (223.11,329.36) ;
\draw [color={rgb, 255:red, 208; green, 2; blue, 27 }  ,draw opacity=1 ]   (223.11,329.36) .. controls (254.59,337.9) and (263.04,338.75) .. (289.86,327.67) ;
\draw [color={rgb, 255:red, 208; green, 2; blue, 27 }  ,draw opacity=1 ]   (256.48,328.52) .. controls (278.02,316.88) and (278.02,312.66) .. (289.86,292.18) ;
\draw [color={rgb, 255:red, 208; green, 2; blue, 27 }  ,draw opacity=1 ]   (223.11,329.36) .. controls (244.65,317.73) and (244.65,313.5) .. (256.48,293.03) ;
\draw [color={rgb, 255:red, 208; green, 2; blue, 27 }  ,draw opacity=1 ]   (205.36,242.33) .. controls (190.99,267.87) and (190.99,294.07) .. (206.21,312.46) ;
\draw [color={rgb, 255:red, 208; green, 2; blue, 27 }  ,draw opacity=1 ]   (220.99,227.93) .. controls (227.32,253.56) and (231.55,271.31) .. (221.42,293.87) ;
\draw [color={rgb, 255:red, 208; green, 2; blue, 27 }  ,draw opacity=1 ]   (220.99,227.93) .. controls (243.38,247.65) and (238.31,303.36) .. (223.11,329.36) ;
\draw [color={rgb, 255:red, 208; green, 2; blue, 27 }  ,draw opacity=1 ]   (205.36,242.33) -- (221.42,293.87) ;
\draw [color={rgb, 255:red, 208; green, 2; blue, 27 }  ,draw opacity=1 ]   (223.11,293.87) .. controls (255.21,278.86) and (263.66,283.08) .. (289.86,292.18) ;
\draw [color={rgb, 255:red, 208; green, 2; blue, 27 }  ,draw opacity=1 ]   (222.26,259.23) .. controls (228.6,284.86) and (232.82,302.6) .. (222.69,325.17) ;
\draw[dotted]     (289.86,292.18) -- (272.96,275.28) ;
\draw  [fill={rgb, 255:red, 74; green, 74; blue, 74 }  ,fill opacity=1 ] (224.43,329.36) .. controls (224.43,328.63) and (223.84,328.04) .. (223.11,328.04) .. controls (222.38,328.04) and (221.79,328.63) .. (221.79,329.36) .. controls (221.79,330.09) and (222.38,330.68) .. (223.11,330.68) .. controls (223.84,330.68) and (224.43,330.09) .. (224.43,329.36) -- cycle ;
\draw  [fill={rgb, 255:red, 74; green, 74; blue, 74 }  ,fill opacity=1 ] (289.86,326.83) .. controls (289.86,326.1) and (289.27,325.51) .. (288.54,325.51) .. controls (287.81,325.51) and (287.22,326.1) .. (287.22,326.83) .. controls (287.22,327.56) and (287.81,328.15) .. (288.54,328.15) .. controls (289.27,328.15) and (289.86,327.56) .. (289.86,326.83) -- cycle ;
\draw  [fill={rgb, 255:red, 74; green, 74; blue, 74 }  ,fill opacity=1 ] (257.38,328.89) .. controls (257.38,328.16) and (256.79,327.56) .. (256.06,327.56) .. controls (255.33,327.56) and (254.74,328.16) .. (254.74,328.89) .. controls (254.74,329.61) and (255.33,330.21) .. (256.06,330.21) .. controls (256.79,330.21) and (257.38,329.61) .. (257.38,328.89) -- cycle ;
\draw  [fill={rgb, 255:red, 74; green, 74; blue, 74 }  ,fill opacity=1 ] (275.6,310.77) .. controls (275.6,310.04) and (275.01,309.45) .. (274.28,309.45) .. controls (273.55,309.45) and (272.96,310.04) .. (272.96,310.77) .. controls (272.96,311.5) and (273.55,312.09) .. (274.28,312.09) .. controls (275.01,312.09) and (275.6,311.5) .. (275.6,310.77) -- cycle ;
\draw  [fill={rgb, 255:red, 74; green, 74; blue, 74 }  ,fill opacity=1 ] (242.22,311.62) .. controls (242.22,310.89) and (241.63,310.3) .. (240.9,310.3) .. controls (240.17,310.3) and (239.58,310.89) .. (239.58,311.62) .. controls (239.58,312.35) and (240.17,312.94) .. (240.9,312.94) .. controls (241.63,312.94) and (242.22,312.35) .. (242.22,311.62) -- cycle ;
\draw  [fill={rgb, 255:red, 74; green, 74; blue, 74 }  ,fill opacity=1 ] (207.53,313.78) .. controls (207.53,313.05) and (206.94,312.46) .. (206.21,312.46) .. controls (205.48,312.46) and (204.89,313.05) .. (204.89,313.78) .. controls (204.89,314.51) and (205.48,315.1) .. (206.21,315.1) .. controls (206.94,315.1) and (207.53,314.51) .. (207.53,313.78) -- cycle ;
\draw  [fill={rgb, 255:red, 74; green, 74; blue, 74 }  ,fill opacity=1 ] (290.12,292.96) .. controls (290.12,292.23) and (289.53,291.64) .. (288.8,291.64) .. controls (288.07,291.64) and (287.48,292.23) .. (287.48,292.96) .. controls (287.48,293.69) and (288.07,294.28) .. (288.8,294.28) .. controls (289.53,294.28) and (290.12,293.69) .. (290.12,292.96) -- cycle ;
\draw  [fill={rgb, 255:red, 74; green, 74; blue, 74 }  ,fill opacity=1 ] (224.06,293.87) .. controls (224.06,293.14) and (223.47,292.55) .. (222.74,292.55) .. controls (222.01,292.55) and (221.42,293.14) .. (221.42,293.87) .. controls (221.42,294.6) and (222.01,295.19) .. (222.74,295.19) .. controls (223.47,295.19) and (224.06,294.6) .. (224.06,293.87) -- cycle ;
\draw  [fill={rgb, 255:red, 74; green, 74; blue, 74 }  ,fill opacity=1 ] (256.11,293.93) .. controls (256.11,293.2) and (255.52,292.61) .. (254.79,292.61) .. controls (254.06,292.61) and (253.47,293.2) .. (253.47,293.93) .. controls (253.47,294.66) and (254.06,295.25) .. (254.79,295.25) .. controls (255.52,295.25) and (256.11,294.66) .. (256.11,293.93) -- cycle ;
\draw  [fill={rgb, 255:red, 74; green, 74; blue, 74 }  ,fill opacity=1 ] (289.06,259.56) .. controls (289.06,258.83) and (288.47,258.24) .. (287.74,258.24) .. controls (287.01,258.24) and (286.42,258.83) .. (286.42,259.56) .. controls (286.42,260.29) and (287.01,260.88) .. (287.74,260.88) .. controls (288.47,260.88) and (289.06,260.29) .. (289.06,259.56) -- cycle ;
\draw  [fill={rgb, 255:red, 74; green, 74; blue, 74 }  ,fill opacity=1 ] (239.21,277.45) .. controls (239.21,276.72) and (238.62,276.13) .. (237.89,276.13) .. controls (237.16,276.13) and (236.57,276.72) .. (236.57,277.45) .. controls (236.57,278.18) and (237.16,278.77) .. (237.89,278.77) .. controls (238.62,278.77) and (239.21,278.18) .. (239.21,277.45) -- cycle ;
\draw  [fill={rgb, 255:red, 74; green, 74; blue, 74 }  ,fill opacity=1 ] (207.11,278.72) .. controls (207.11,277.99) and (206.51,277.4) .. (205.79,277.4) .. controls (205.06,277.4) and (204.46,277.99) .. (204.46,278.72) .. controls (204.46,279.45) and (205.06,280.04) .. (205.79,280.04) .. controls (206.51,280.04) and (207.11,279.45) .. (207.11,278.72) -- cycle ;
\draw  [fill={rgb, 255:red, 74; green, 74; blue, 74 }  ,fill opacity=1 ] (224.43,261.39) .. controls (224.43,260.67) and (223.84,260.07) .. (223.11,260.07) .. controls (222.38,260.07) and (221.79,260.67) .. (221.79,261.39) .. controls (221.79,262.12) and (222.38,262.72) .. (223.11,262.72) .. controls (223.84,262.72) and (224.43,262.12) .. (224.43,261.39) -- cycle ;
\draw  [fill={rgb, 255:red, 74; green, 74; blue, 74 }  ,fill opacity=1 ] (274.28,276.6) .. controls (274.28,275.87) and (273.69,275.28) .. (272.96,275.28) .. controls (272.23,275.28) and (271.64,275.87) .. (271.64,276.6) .. controls (271.64,277.33) and (272.23,277.92) .. (272.96,277.92) .. controls (273.69,277.92) and (274.28,277.33) .. (274.28,276.6) -- cycle ;
\draw  [fill={rgb, 255:red, 74; green, 74; blue, 74 }  ,fill opacity=1 ] (257.8,260.55) .. controls (257.8,259.82) and (257.21,259.23) .. (256.48,259.23) .. controls (255.75,259.23) and (255.16,259.82) .. (255.16,260.55) .. controls (255.16,261.28) and (255.75,261.87) .. (256.48,261.87) .. controls (257.21,261.87) and (257.8,261.28) .. (257.8,260.55) -- cycle ;
\draw  [fill={rgb, 255:red, 74; green, 74; blue, 74 }  ,fill opacity=1 ] (208,242.33) .. controls (208,241.6) and (207.41,241.01) .. (206.68,241.01) .. controls (205.95,241.01) and (205.36,241.6) .. (205.36,242.33) .. controls (205.36,243.06) and (205.95,243.65) .. (206.68,243.65) .. controls (207.41,243.65) and (208,243.06) .. (208,242.33) -- cycle ;
\draw  [fill={rgb, 255:red, 74; green, 74; blue, 74 }  ,fill opacity=1 ] (240.06,242.81) .. controls (240.06,242.08) and (239.47,241.48) .. (238.74,241.48) .. controls (238.01,241.48) and (237.42,242.08) .. (237.42,242.81) .. controls (237.42,243.53) and (238.01,244.13) .. (238.74,244.13) .. controls (239.47,244.13) and (240.06,243.53) .. (240.06,242.81) -- cycle ;
\draw  [fill={rgb, 255:red, 74; green, 74; blue, 74 }  ,fill opacity=1 ] (268.73,241.34) .. controls (268.73,242.51) and (269.67,243.46) .. (270.84,243.46) .. controls (272.01,243.46) and (272.95,242.51) .. (272.95,241.34) .. controls (272.95,240.18) and (272.01,239.23) .. (270.84,239.23) .. controls (269.67,239.23) and (268.73,240.18) .. (268.73,241.34) -- cycle ;
\draw[dotted]     (287.74,258.24) -- (270.84,241.34) ;
\draw  [fill={rgb, 255:red, 74; green, 74; blue, 74 }  ,fill opacity=1 ] (287.74,226.24) .. controls (287.74,227) and (288.36,227.62) .. (289.12,227.62) .. controls (289.88,227.62) and (290.5,227) .. (290.5,226.24) .. controls (290.5,225.48) and (289.88,224.86) .. (289.12,224.86) .. controls (288.36,224.86) and (287.74,225.48) .. (287.74,226.24) -- cycle ;
\draw  [fill={rgb, 255:red, 74; green, 74; blue, 74 }  ,fill opacity=1 ] (254.36,227.09) .. controls (254.36,227.85) and (254.98,228.47) .. (255.74,228.47) .. controls (256.51,228.47) and (257.12,227.85) .. (257.12,227.09) .. controls (257.12,226.32) and (256.51,225.71) .. (255.74,225.71) .. controls (254.98,225.71) and (254.36,226.32) .. (254.36,227.09) -- cycle ;
\draw  [fill={rgb, 255:red, 74; green, 74; blue, 74 }  ,fill opacity=1 ] (201.44,209.88) .. controls (201.44,210.64) and (202.06,211.26) .. (202.82,211.26) .. controls (203.58,211.26) and (204.2,210.64) .. (204.2,209.88) .. controls (204.2,209.12) and (203.58,208.5) .. (202.82,208.5) .. controls (202.06,208.5) and (201.44,209.12) .. (201.44,209.88) -- cycle ;
\draw  [fill={rgb, 255:red, 74; green, 74; blue, 74 }  ,fill opacity=1 ] (220.99,227.93) .. controls (220.99,228.69) and (221.6,229.31) .. (222.37,229.31) .. controls (223.13,229.31) and (223.75,228.69) .. (223.75,227.93) .. controls (223.75,227.17) and (223.13,226.55) .. (222.37,226.55) .. controls (221.6,226.55) and (220.99,227.17) .. (220.99,227.93) -- cycle ;
\draw  [fill={rgb, 255:red, 74; green, 74; blue, 74 }  ,fill opacity=1 ] (269.88,209.03) .. controls (269.88,209.8) and (270.5,210.41) .. (271.26,210.41) .. controls (272.02,210.41) and (272.64,209.8) .. (272.64,209.03) .. controls (272.64,208.27) and (272.02,207.65) .. (271.26,207.65) .. controls (270.5,207.65) and (269.88,208.27) .. (269.88,209.03) -- cycle ;
\draw  [fill={rgb, 255:red, 74; green, 74; blue, 74 }  ,fill opacity=1 ] (237.89,208.5) .. controls (237.89,209.26) and (238.5,209.88) .. (239.27,209.88) .. controls (240.03,209.88) and (240.65,209.26) .. (240.65,208.5) .. controls (240.65,207.73) and (240.03,207.12) .. (239.27,207.12) .. controls (238.5,207.12) and (237.89,207.73) .. (237.89,208.5) -- cycle ;

\end{tikzpicture}
\caption*{An example of a saturated transfer system $S$ on \(([1]\times [2])\times [3]\)}
\end{figure}
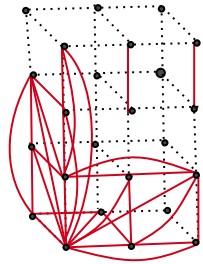

\begin{figure}[h]
\centering
\begin{subfigure}{5cm}
\tikzset{every picture/.style={line width=0.75pt}} 

\begin{tikzpicture}[x=0.75pt,y=0.75pt,yscale=-1,xscale=1]

\draw[dotted]    (502.32,192.5) -- (504.86,226.33) ;
\draw[dotted]    (520.91,210.24) -- (522.61,313.36) ;
\draw[dotted]    (553.86,245.73) -- (554.29,277.03) ;
\draw[dotted]    (589.36,242.38) -- (589.36,311.67) ;
\draw[dotted]    (570.76,191.65) -- (572.46,294.77) ;
\draw[dotted]    (536.96,190.81) -- (538.66,293.93) ;
\draw[dotted]    (587.24,210.24) -- (520.49,211.93) ;
\draw[dotted]    (570.76,191.65) -- (504.01,193.34) ;
\draw[dotted]    (572.46,294.77) -- (505.71,296.46) ;
\draw[dotted]    (589.36,242.38) -- (522.61,244.07) ;
\draw[dotted]    (570.77,259.28) -- (504.02,260.97) ;
\draw[dotted]    (571.61,224.64) -- (504.86,226.33) ;
\draw[dotted]    (522.61,313.36) -- (505.71,296.46) ;
\draw[dotted]    (520.49,211.93) -- (503.59,195.03) ;
\draw[dotted]    (587.66,208.55) -- (570.76,191.65) ;
\draw[dotted]    (554.28,209.4) -- (537.39,192.5) ;
\draw[dotted]    (589.36,311.67) -- (572.46,294.77) ;
\draw[dotted]    (554.29,277.93) -- (537.39,261.03) ;
\draw[dotted]    (555.98,243.23) -- (539.08,226.33) ;
\draw[dotted]   (589.36,276.18) -- (589.36,310.83) ;
\draw [color={rgb, 255:red, 208; green, 2; blue, 27 }  ,draw opacity=1 ]   (505.71,296.46) -- (513.29,304.05) -- (522.61,313.36) ;
\draw [color={rgb, 255:red, 208; green, 2; blue, 27 }  ,draw opacity=1 ]   (589.36,310.83) -- (522.61,313.36) ;
\draw[dotted]    (554.29,276.61) -- (555.56,314.21) ;
\draw [color={rgb, 255:red, 74; green, 144; blue, 226 }  ,draw opacity=1 ]   (504.02,260.97) -- (520.92,277.87) ;
\draw [color={rgb, 255:red, 245; green, 166; blue, 35 }  ,draw opacity=1 ]   (504.86,226.33) -- (521.76,243.23) ;
\draw [color={rgb, 255:red, 208; green, 2; blue, 27 }  ,draw opacity=1 ]   (539.08,295.62) -- (505.71,296.46) ;
\draw [color={rgb, 255:red, 208; green, 2; blue, 27 }  ,draw opacity=1 ]   (539.51,293.93) -- (555.98,312.52) ;
\draw [color={rgb, 255:red, 74; green, 144; blue, 226 }  ,draw opacity=1 ]   (587.67,275.34) -- (520.92,275.34) ;
\draw [color={rgb, 255:red, 208; green, 2; blue, 27 }  ,draw opacity=1 ]   (522.61,313.36) .. controls (554.09,321.9) and (562.54,322.75) .. (589.36,311.67) ;
\draw [color={rgb, 255:red, 74; green, 144; blue, 226 }  ,draw opacity=1 ]   (522.61,277.87) .. controls (554.71,262.86) and (563.16,267.08) .. (589.36,276.18) ;
\draw[dotted]    (589.36,276.18) -- (572.46,259.28) ;
\draw  [fill={rgb, 255:red, 74; green, 74; blue, 74 }  ,fill opacity=1 ] (523.93,313.36) .. controls (523.93,312.63) and (523.34,312.04) .. (522.61,312.04) .. controls (521.88,312.04) and (521.29,312.63) .. (521.29,313.36) .. controls (521.29,314.09) and (521.88,314.68) .. (522.61,314.68) .. controls (523.34,314.68) and (523.93,314.09) .. (523.93,313.36) -- cycle ;
\draw  [fill={rgb, 255:red, 74; green, 74; blue, 74 }  ,fill opacity=1 ] (589.36,310.83) .. controls (589.36,310.1) and (588.77,309.51) .. (588.04,309.51) .. controls (587.31,309.51) and (586.72,310.1) .. (586.72,310.83) .. controls (586.72,311.56) and (587.31,312.15) .. (588.04,312.15) .. controls (588.77,312.15) and (589.36,311.56) .. (589.36,310.83) -- cycle ;
\draw  [fill={rgb, 255:red, 74; green, 74; blue, 74 }  ,fill opacity=1 ] (556.88,312.89) .. controls (556.88,312.16) and (556.29,311.56) .. (555.56,311.56) .. controls (554.83,311.56) and (554.24,312.16) .. (554.24,312.89) .. controls (554.24,313.61) and (554.83,314.21) .. (555.56,314.21) .. controls (556.29,314.21) and (556.88,313.61) .. (556.88,312.89) -- cycle ;
\draw  [fill={rgb, 255:red, 74; green, 74; blue, 74 }  ,fill opacity=1 ] (575.1,294.77) .. controls (575.1,294.04) and (574.51,293.45) .. (573.78,293.45) .. controls (573.05,293.45) and (572.46,294.04) .. (572.46,294.77) .. controls (572.46,295.5) and (573.05,296.09) .. (573.78,296.09) .. controls (574.51,296.09) and (575.1,295.5) .. (575.1,294.77) -- cycle ;
\draw  [fill={rgb, 255:red, 74; green, 74; blue, 74 }  ,fill opacity=1 ] (541.72,295.62) .. controls (541.72,294.89) and (541.13,294.3) .. (540.4,294.3) .. controls (539.67,294.3) and (539.08,294.89) .. (539.08,295.62) .. controls (539.08,296.35) and (539.67,296.94) .. (540.4,296.94) .. controls (541.13,296.94) and (541.72,296.35) .. (541.72,295.62) -- cycle ;
\draw  [fill={rgb, 255:red, 74; green, 74; blue, 74 }  ,fill opacity=1 ] (507.03,297.78) .. controls (507.03,297.05) and (506.44,296.46) .. (505.71,296.46) .. controls (504.98,296.46) and (504.39,297.05) .. (504.39,297.78) .. controls (504.39,298.51) and (504.98,299.1) .. (505.71,299.1) .. controls (506.44,299.1) and (507.03,298.51) .. (507.03,297.78) -- cycle ;
\draw  [fill={rgb, 255:red, 74; green, 74; blue, 74 }  ,fill opacity=1 ] (589.62,276.96) .. controls (589.62,276.23) and (589.03,275.64) .. (588.3,275.64) .. controls (587.57,275.64) and (586.98,276.23) .. (586.98,276.96) .. controls (586.98,277.69) and (587.57,278.28) .. (588.3,278.28) .. controls (589.03,278.28) and (589.62,277.69) .. (589.62,276.96) -- cycle ;
\draw  [fill={rgb, 255:red, 74; green, 74; blue, 74 }  ,fill opacity=1 ] (523.56,277.87) .. controls (523.56,277.14) and (522.97,276.55) .. (522.24,276.55) .. controls (521.51,276.55) and (520.92,277.14) .. (520.92,277.87) .. controls (520.92,278.6) and (521.51,279.19) .. (522.24,279.19) .. controls (522.97,279.19) and (523.56,278.6) .. (523.56,277.87) -- cycle ;
\draw  [fill={rgb, 255:red, 74; green, 74; blue, 74 }  ,fill opacity=1 ] (555.61,277.93) .. controls (555.61,277.2) and (555.02,276.61) .. (554.29,276.61) .. controls (553.56,276.61) and (552.97,277.2) .. (552.97,277.93) .. controls (552.97,278.66) and (553.56,279.25) .. (554.29,279.25) .. controls (555.02,279.25) and (555.61,278.66) .. (555.61,277.93) -- cycle ;
\draw  [fill={rgb, 255:red, 74; green, 74; blue, 74 }  ,fill opacity=1 ] (590.68,243.7) .. controls (590.68,242.98) and (590.09,242.38) .. (589.36,242.38) .. controls (588.63,242.38) and (588.04,242.98) .. (588.04,243.7) .. controls (588.04,244.43) and (588.63,245.03) .. (589.36,245.03) .. controls (590.09,245.03) and (590.68,244.43) .. (590.68,243.7) -- cycle ;
\draw  [fill={rgb, 255:red, 74; green, 74; blue, 74 }  ,fill opacity=1 ] (538.71,261.45) .. controls (538.71,260.72) and (538.12,260.13) .. (537.39,260.13) .. controls (536.66,260.13) and (536.07,260.72) .. (536.07,261.45) .. controls (536.07,262.18) and (536.66,262.77) .. (537.39,262.77) .. controls (538.12,262.77) and (538.71,262.18) .. (538.71,261.45) -- cycle ;
\draw  [fill={rgb, 255:red, 74; green, 74; blue, 74 }  ,fill opacity=1 ] (506.61,262.72) .. controls (506.61,261.99) and (506.01,261.4) .. (505.29,261.4) .. controls (504.56,261.4) and (503.96,261.99) .. (503.96,262.72) .. controls (503.96,263.45) and (504.56,264.04) .. (505.29,264.04) .. controls (506.01,264.04) and (506.61,263.45) .. (506.61,262.72) -- cycle ;
\draw  [fill={rgb, 255:red, 74; green, 74; blue, 74 }  ,fill opacity=1 ] (523.93,245.39) .. controls (523.93,244.67) and (523.34,244.07) .. (522.61,244.07) .. controls (521.88,244.07) and (521.29,244.67) .. (521.29,245.39) .. controls (521.29,246.12) and (521.88,246.72) .. (522.61,246.72) .. controls (523.34,246.72) and (523.93,246.12) .. (523.93,245.39) -- cycle ;
\draw  [fill={rgb, 255:red, 74; green, 74; blue, 74 }  ,fill opacity=1 ] (573.78,260.6) .. controls (573.78,259.87) and (573.19,259.28) .. (572.46,259.28) .. controls (571.73,259.28) and (571.14,259.87) .. (571.14,260.6) .. controls (571.14,261.33) and (571.73,261.92) .. (572.46,261.92) .. controls (573.19,261.92) and (573.78,261.33) .. (573.78,260.6) -- cycle ;
\draw  [fill={rgb, 255:red, 74; green, 74; blue, 74 }  ,fill opacity=1 ] (557.3,244.55) .. controls (557.3,243.82) and (556.71,243.23) .. (555.98,243.23) .. controls (555.25,243.23) and (554.66,243.82) .. (554.66,244.55) .. controls (554.66,245.28) and (555.25,245.87) .. (555.98,245.87) .. controls (556.71,245.87) and (557.3,245.28) .. (557.3,244.55) -- cycle ;
\draw  [fill={rgb, 255:red, 74; green, 74; blue, 74 }  ,fill opacity=1 ] (507.5,226.33) .. controls (507.5,225.6) and (506.91,225.01) .. (506.18,225.01) .. controls (505.45,225.01) and (504.86,225.6) .. (504.86,226.33) .. controls (504.86,227.06) and (505.45,227.65) .. (506.18,227.65) .. controls (506.91,227.65) and (507.5,227.06) .. (507.5,226.33) -- cycle ;
\draw  [fill={rgb, 255:red, 74; green, 74; blue, 74 }  ,fill opacity=1 ] (539.56,226.81) .. controls (539.56,226.08) and (538.97,225.48) .. (538.24,225.48) .. controls (537.51,225.48) and (536.92,226.08) .. (536.92,226.81) .. controls (536.92,227.53) and (537.51,228.13) .. (538.24,228.13) .. controls (538.97,228.13) and (539.56,227.53) .. (539.56,226.81) -- cycle ;
\draw  [fill={rgb, 255:red, 74; green, 74; blue, 74 }  ,fill opacity=1 ] (568.23,225.34) .. controls (568.23,226.51) and (569.17,227.46) .. (570.34,227.46) .. controls (571.51,227.46) and (572.45,226.51) .. (572.45,225.34) .. controls (572.45,224.18) and (571.51,223.23) .. (570.34,223.23) .. controls (569.17,223.23) and (568.23,224.18) .. (568.23,225.34) -- cycle ;
\draw[dotted]    (587.24,242.24) -- (570.34,225.34) ;
\draw  [fill={rgb, 255:red, 74; green, 74; blue, 74 }  ,fill opacity=1 ] (587.24,210.24) .. controls (587.24,211) and (587.86,211.62) .. (588.62,211.62) .. controls (589.38,211.62) and (590,211) .. (590,210.24) .. controls (590,209.48) and (589.38,208.86) .. (588.62,208.86) .. controls (587.86,208.86) and (587.24,209.48) .. (587.24,210.24) -- cycle ;
\draw  [fill={rgb, 255:red, 74; green, 74; blue, 74 }  ,fill opacity=1 ] (553.86,211.09) .. controls (553.86,211.85) and (554.48,212.47) .. (555.24,212.47) .. controls (556.01,212.47) and (556.62,211.85) .. (556.62,211.09) .. controls (556.62,210.32) and (556.01,209.71) .. (555.24,209.71) .. controls (554.48,209.71) and (553.86,210.32) .. (553.86,211.09) -- cycle ;
\draw  [fill={rgb, 255:red, 74; green, 74; blue, 74 }  ,fill opacity=1 ] (500.94,193.88) .. controls (500.94,194.64) and (501.56,195.26) .. (502.32,195.26) .. controls (503.08,195.26) and (503.7,194.64) .. (503.7,193.88) .. controls (503.7,193.12) and (503.08,192.5) .. (502.32,192.5) .. controls (501.56,192.5) and (500.94,193.12) .. (500.94,193.88) -- cycle ;
\draw  [fill={rgb, 255:red, 74; green, 74; blue, 74 }  ,fill opacity=1 ] (520.49,211.93) .. controls (520.49,212.69) and (521.1,213.31) .. (521.87,213.31) .. controls (522.63,213.31) and (523.25,212.69) .. (523.25,211.93) .. controls (523.25,211.17) and (522.63,210.55) .. (521.87,210.55) .. controls (521.1,210.55) and (520.49,211.17) .. (520.49,211.93) -- cycle ;
\draw  [fill={rgb, 255:red, 74; green, 74; blue, 74 }  ,fill opacity=1 ] (569.38,193.03) .. controls (569.38,193.8) and (570,194.41) .. (570.76,194.41) .. controls (571.52,194.41) and (572.14,193.8) .. (572.14,193.03) .. controls (572.14,192.27) and (571.52,191.65) .. (570.76,191.65) .. controls (570,191.65) and (569.38,192.27) .. (569.38,193.03) -- cycle ;
\draw  [fill={rgb, 255:red, 74; green, 74; blue, 74 }  ,fill opacity=1 ] (537.39,192.5) .. controls (537.39,193.26) and (538,193.88) .. (538.77,193.88) .. controls (539.53,193.88) and (540.15,193.26) .. (540.15,192.5) .. controls (540.15,191.73) and (539.53,191.12) .. (538.77,191.12) .. controls (538,191.12) and (537.39,191.73) .. (537.39,192.5) -- cycle ;
\draw[dotted]    (504.86,226.33) -- (505.71,297.78) ;
\draw[dotted]    (588.62,210.24) -- (589.36,242.38) ;
\draw[dotted]    (555.24,212.47) -- (555.98,244.55) ;
\end{tikzpicture}
\caption*{Layers \(S_0\), \(S_1\), \(S_2\) are dawn in red, blue, and orange respectively. Note that \(S_3=\emptyset\). }
\end{subfigure}
\begin{subfigure}{5cm}
\centering
\tikzset{every picture/.style={line width=0.75pt}} 

\begin{tikzpicture}[x=0.75pt,y=0.75pt,yscale=-1,xscale=1]

\draw[dotted]    (478.46,66.81) -- (479.74,101.48) ;
\draw[dotted]    (528.74,86.24) -- (461.99,87.93) ;
\draw[dotted]    (514.43,69.08) -- (447.67,70.77) ;
\draw[dotted]    (513.11,100.64) -- (446.36,102.33) ;
\draw[dotted]    (529.16,84.55) -- (512.26,67.65) ;
\draw[dotted]    (495.78,85.4) -- (478.89,68.5) ;
\draw  [fill={rgb, 255:red, 74; green, 74; blue, 74 }  ,fill opacity=1 ] (449,102.33) .. controls (449,101.6) and (448.41,101.01) .. (447.68,101.01) .. controls (446.95,101.01) and (446.36,101.6) .. (446.36,102.33) .. controls (446.36,103.06) and (446.95,103.65) .. (447.68,103.65) .. controls (448.41,103.65) and (449,103.06) .. (449,102.33) -- cycle ;
\draw  [fill={rgb, 255:red, 74; green, 74; blue, 74 }  ,fill opacity=1 ] (481.06,102.81) .. controls (481.06,102.08) and (480.47,101.48) .. (479.74,101.48) .. controls (479.01,101.48) and (478.42,102.08) .. (478.42,102.81) .. controls (478.42,103.53) and (479.01,104.13) .. (479.74,104.13) .. controls (480.47,104.13) and (481.06,103.53) .. (481.06,102.81) -- cycle ;
\draw[dotted]    (528.74,118.24) -- (513.11,100.64) ;
\draw  [fill={rgb, 255:red, 74; green, 74; blue, 74 }  ,fill opacity=1 ] (528.74,86.24) .. controls (528.74,87) and (529.36,87.62) .. (530.12,87.62) .. controls (530.88,87.62) and (531.5,87) .. (531.5,86.24) .. controls (531.5,85.48) and (530.88,84.86) .. (530.12,84.86) .. controls (529.36,84.86) and (528.74,85.48) .. (528.74,86.24) -- cycle ;
\draw  [fill={rgb, 255:red, 74; green, 74; blue, 74 }  ,fill opacity=1 ] (495.36,87.09) .. controls (495.36,87.85) and (495.98,88.47) .. (496.74,88.47) .. controls (497.51,88.47) and (498.12,87.85) .. (498.12,87.09) .. controls (498.12,86.32) and (497.51,85.71) .. (496.74,85.71) .. controls (495.98,85.71) and (495.36,86.32) .. (495.36,87.09) -- cycle ;
\draw  [fill={rgb, 255:red, 74; green, 74; blue, 74 }  ,fill opacity=1 ] (446.29,70.77) .. controls (446.29,71.53) and (446.91,72.15) .. (447.67,72.15) .. controls (448.44,72.15) and (449.06,71.53) .. (449.06,70.77) .. controls (449.06,70.01) and (448.44,69.39) .. (447.67,69.39) .. controls (446.91,69.39) and (446.29,70.01) .. (446.29,70.77) -- cycle ;
\draw  [fill={rgb, 255:red, 74; green, 74; blue, 74 }  ,fill opacity=1 ] (461.99,87.93) .. controls (461.99,88.69) and (462.6,89.31) .. (463.37,89.31) .. controls (464.13,89.31) and (464.75,88.69) .. (464.75,87.93) .. controls (464.75,87.17) and (464.13,86.55) .. (463.37,86.55) .. controls (462.6,86.55) and (461.99,87.17) .. (461.99,87.93) -- cycle ;
\draw  [fill={rgb, 255:red, 74; green, 74; blue, 74 }  ,fill opacity=1 ] (510.88,69.03) .. controls (510.88,69.8) and (511.5,70.41) .. (512.26,70.41) .. controls (513.02,70.41) and (513.64,69.8) .. (513.64,69.03) .. controls (513.64,68.27) and (513.02,67.65) .. (512.26,67.65) .. controls (511.5,67.65) and (510.88,68.27) .. (510.88,69.03) -- cycle ;
\draw  [fill={rgb, 255:red, 74; green, 74; blue, 74 }  ,fill opacity=1 ] (478.89,68.5) .. controls (478.89,69.26) and (479.5,69.88) .. (480.27,69.88) .. controls (481.03,69.88) and (481.65,69.26) .. (481.65,68.5) .. controls (481.65,67.73) and (481.03,67.12) .. (480.27,67.12) .. controls (479.5,67.12) and (478.89,67.73) .. (478.89,68.5) -- cycle ;
\draw[dotted]    (530.12,86.24) -- (530.86,118.38) ;
\draw[dotted]    (496.74,88.47) -- (497.48,120.55) ;
\draw[dotted]    (512.26,70.41) -- (513,102.5) ;
\draw [color={rgb, 255:red, 208; green, 2; blue, 27 }  ,draw opacity=1 ]   (530.86,118.38) -- (464.11,120.92) ;
\draw  [fill={rgb, 255:red, 74; green, 74; blue, 74 }  ,fill opacity=1 ] (464.58,120.55) .. controls (464.58,119.82) and (463.99,119.23) .. (463.26,119.23) .. controls (462.53,119.23) and (461.94,119.82) .. (461.94,120.55) .. controls (461.94,121.28) and (462.53,121.87) .. (463.26,121.87) .. controls (463.99,121.87) and (464.58,121.28) .. (464.58,120.55) -- cycle ;
\draw [color={rgb, 255:red, 208; green, 2; blue, 27 }  ,draw opacity=1 ]   (463.37,86.55) -- (463.26,120.55) ;
\draw [color={rgb, 255:red, 208; green, 2; blue, 27 }  ,draw opacity=1 ]   (447.68,101.01) -- (463.26,120.55) ;
\draw [color={rgb, 255:red, 208; green, 2; blue, 27 }  ,draw opacity=1 ]   (463.26,121.87) .. controls (494.74,130.41) and (503.83,126.59) .. (530.86,118.38) ;
\draw [color={rgb, 255:red, 208; green, 2; blue, 27 }  ,draw opacity=1 ]   (447.67,72.15) -- (447.68,102.33) ;
\draw [color={rgb, 255:red, 208; green, 2; blue, 27 }  ,draw opacity=1 ]   (449.06,70.77) -- (463.37,87.93) ;
\draw  [fill={rgb, 255:red, 74; green, 74; blue, 74 }  ,fill opacity=1 ] (498.8,119.23) .. controls (498.8,118.5) and (498.21,117.91) .. (497.48,117.91) .. controls (496.75,117.91) and (496.16,118.5) .. (496.16,119.23) .. controls (496.16,119.96) and (496.75,120.55) .. (497.48,120.55) .. controls (498.21,120.55) and (498.8,119.96) .. (498.8,119.23) -- cycle ;
\draw  [fill={rgb, 255:red, 74; green, 74; blue, 74 }  ,fill opacity=1 ] (531.38,118.24) .. controls (531.38,117.51) and (530.79,116.92) .. (530.06,116.92) .. controls (529.33,116.92) and (528.74,117.51) .. (528.74,118.24) .. controls (528.74,118.97) and (529.33,119.56) .. (530.06,119.56) .. controls (530.79,119.56) and (531.38,118.97) .. (531.38,118.24) -- cycle ;
\draw  [fill={rgb, 255:red, 74; green, 74; blue, 74 }  ,fill opacity=1 ] (513.11,100.64) .. controls (513.11,99.91) and (512.52,99.32) .. (511.79,99.32) .. controls (511.06,99.32) and (510.47,99.91) .. (510.47,100.64) .. controls (510.47,101.37) and (511.06,101.96) .. (511.79,101.96) .. controls (512.52,101.96) and (513.11,101.37) .. (513.11,100.64) -- cycle ;
\draw [color={rgb, 255:red, 208; green, 2; blue, 27 }  ,draw opacity=1 ]   (447.67,70.77) -- (463.26,121.87) ;
\draw[dotted]    (497.48,119.65) -- (479.74,104.13) ;

\end{tikzpicture}

\caption*{Picture of \(S^2\)}
\end{subfigure}
\caption{Diagrams depicting \autoref{defn:layer} when $P=\chain 1\times\chain 2$ and $S$ is the saturated transfer system on $P\times \chain 3$ in the top diagram.}\label{fig:layer}
\end{figure}

We start with an observation about saturated transfers systems on $P \times [1]$, which follows from the fact that the restriction of $(x,1)\to (y,1)$ along $(y,0)$ is $(x,0)\to (y,0)$.

\begin{lemma}\label{lem:layer-ordering}
    If $S$ is a saturated transfer system on $P\times [1]$, then the layers satisfy $S_1 \leq S_0$.\hfill\qedsymbol
\end{lemma}

For a general saturated transfer system on $P \times [n]$, it turns out that the information about height one cylinders is all that is needed to recover the entire system.

\begin{lemma}\label{lem:slice-decomposition}
  Let $(P,\leq)$ be a finite lattice, and $S$ a saturated transfer system on $P\times [n]$. Then $S$ is completely determined by the tuple $(S^1,\dots,S^n)$ of saturated transfer systems on $P\times [1]$.
\end{lemma}

\begin{proof}
   Given a saturated transfer system $S$, we consider the relation $Q$ on $P\times [n]$ given by taking $\bigcup_{i=1}^n S^i$ and then closing under transitivity. We claim $Q=S$, and thus, we can uniquely recover $S$ from its height one cylinders.

   Note that since $S$ contains $\bigcup_{i-1}^n S^i$ and it is closed under transitivity, then $Q \subseteq S$. On the other hand, assume $(x,i)S(y,j)$, for some $x\leq y$ in $P$, and $0\leq i \leq j \leq n$. Since $S$ is a saturated transfer system, it is decomposable, and thus we have $(x,i)S(y,i)$, and $(y,k-1)S(y,k)$ for all $k=i+1,\dots,j$. All these relations are in $\bigcup_{i-1}^n S^i$, and thus $(x,i)Q(y,j)$, proving that $S\subseteq Q$.
\end{proof}

Following the same strategy as in \autoref{sec:submonoids}, we construct a transition matrix that will aid in the enumeration of saturated transfer systems on $P\times [n]$. We remark that in this case the proofs are easier, in light of \autoref{lem:slice-decomposition}.

\begin{defn}\label{defn:saturated-graph}
    Let $(P,\leq)$ be a finite lattice. The \emph{saturated transfer system graph} of $P$ is the weighted directed graph $ST(P)$ with vertices the saturated transfer systems of $P$ and weight function $w \colon \SatTr(P) \times \SatTr(P) \to \NN$ given by
    \[w(R,Q):= \# \{ S \in \SatTr(P\times [1]) : S_0=Q \text{ and } S_1=R\}.\]

    After choosing a linear ordering on $\SatTr(P)$, we write $W=W(P)$ for the associated adjacency matrix with $W_{R,Q}=w(R,Q)$.
\end{defn}

\begin{remark}
    The result of \autoref{lem:layer-ordering} implies that if $w(R,Q)\neq 0$, then $R\leq Q$ in the refinement order of $\SatTr(P)$.  The converse is also true, as if $R\leq Q$, we can always construct the transfer system on $P \times [1]$ given by placing $Q$ on the $P\times \{0\}$ layer, $R$ on the $P\times \{1\}$ layer, and including no vertical relations. 

    Thus we can think of the graph as enhancing the directed graph given by the poset $(\SatTr(P),\leq)$ with positive weights.
\end{remark}

\begin{prop}
    Let $P$ be a finite lattice with weighted adjacency matrix $W=W(P)$. Then for all $R,Q\in \SatTr(P)$, and all $n\geq 0$,
  \[\#\{S \in \SatTr(P\times [n]) : S_0=Q \text{ and } S_n=R\}=(W^n)_{R,Q}.\]
  Thus,
  \[\# \SatTr(P\times [n]) =\sum_{R,Q\in \SatTr(P)}(W^n)_{R,Q}.\]
\end{prop}

\begin{proof}
    The second equality follows directly from the first. To prove the first assertion, note that \autoref{lem:slice-decomposition} implies that a saturated transfer system is determined by a path $S_n\leq S_{n-1} \leq \dots \leq S_0$ in the poset $\SatTr(P)$, and for each such path, a choice of how to fill the cylinders $S^i$, with the condition that the bottom and top layers of $S^i$ are $S_{i-1}$ and $S_i$, respectively, which is precisely being counted by $W_{S_i,S_{i-1}}$. The entries of the powers of the adjacency matrix are precisely giving the count when $S_0$ and $S_n$ are fixed.
\end{proof}

We finish this section by proving the following result. Recall the bijection $\chi$ of \autoref{thm:bij-sattr-subm}, and the weighted graphs $G(P)$ of \autoref{defn:submon_graph} and $ST(P)$ of \autoref{defn:saturated-graph}.

\begin{theorem}
    For a finite lattice $P$, the bijection $\chi$ induces an isomorphism of weighted graphs between $ST(P)$ and $G(P)$.
\end{theorem}

\begin{proof}
    Since the vertices of $ST(P)$ are the saturated transfer systems of $P$ and the vertices of $G(P)$ are the submonoids of $(P, \vee)$, \autoref{thm:bij-sattr-subm} establishes a bijection between both sets of vertices. It thus remains to show that for saturated transfer systems $R,Q$ of $P$, 
    \[w(R,Q)=w(\chi(R),\chi(Q)).\]
    Recall that the left-hand side is counting the saturated transfer systems on $P\times [1]$ with bottom layer given by $Q$ and top layer given by $R$, while the right-hand side is counting the ideals (\emph{i.e.}, upward-closed sets) of $\chi(R)$ whose union with $\chi(Q)$ is equal to $\chi(R)$. We thus construct an explicit bijection between these sets. 

    To go forward, let $S$ be a saturated transfer system on $P\times [1]$ such that $S_0=Q$ and $S_1=R$. Let 
    \[V=\{ x \in P : (x,0) S (x,1)\},\]
    that is, the set of elements in $P$ that correspond to vertical arrows in $S$. 
    Now consider
    the set $I=\chi(R) \cap (P\setminus V)$. The restriction condition on $S$ implies $V$ is downward-closed, and hence $I$ is an upward-closed subset of $\chi(R)$. We claim that $I\cup \chi(Q)=\chi(R)$. Note that since $S$ exists, $R\leq Q$, and since $\chi$ is order reversing, we have that $\chi(Q) \subseteq \chi (R)$, giving one containment. For the reverse containment, suppose $x\in \chi(R)$ but $x\not\in I$. This means by definition that $x\in V$, so $(x,0)S(x,1)$. Since $x\in\chi(R)$, $x$ is minimal in its connected component in $R$. We now show that $x$ is also minimal in its connected component in $Q$, thus showing that $x\in\chi(Q)$, and proving the containment. Let $y\leq x$ in $P$, and suppose for the sake of contradiction that $y\to x$ is in $Q$, meaning that the following are all in $S$.

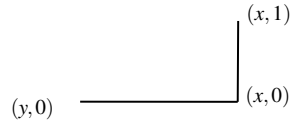
\begin{center}
\tikzset{every picture/.style={line width=0.75pt}} 

\begin{tikzpicture}[x=0.75pt,y=0.75pt,yscale=-1,xscale=1]

\draw    (55.03,150.43) -- (134.5,150.5) ;
\draw    (134.73,109.69) -- (134.5,150.5) ;

\draw (137,141) node [anchor=north west][inner sep=0.75pt]  [font=\scriptsize]  {$( x,0)$};
\draw (137.5,99.5) node [anchor=north west][inner sep=0.75pt]  [font=\scriptsize]  {$( x,1)$};
\draw (18.5,147) node [anchor=north west][inner sep=0.75pt]  [font=\scriptsize]  {$( y,0)$};
\end{tikzpicture}
\end{center}
    \noindent Then we have that $(y,0)\to (x,0)$ and $(x,0)\to(x,1)$ are in $S$, which implies by transitivity that $(y,0) \to (x,1)$ is in $S$. By decomposability, this in turn implies that $(y,1)\to (x,1)$, which contradicts the minimality of $x$ within its connected component in $R$, as depicted below.
    \begin{center}
        \tikzset{every picture/.style={line width=0.75pt}} 

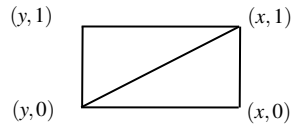
\begin{tikzpicture}[x=0.75pt,y=0.75pt,yscale=-1,xscale=1]

\draw    (239.53,149.43) -- (319,149.5) ;
\draw    (239.77,108.62) -- (319.23,108.69) ;
\draw    (239.53,149.43) -- (319.23,108.69) ;
\draw    (319.73,108.69) -- (319.5,149.5) ;
\draw    (240.27,108.62) -- (240.03,149.43) ;

\draw (321.5,145.5) node [anchor=north west][inner sep=0.75pt]  [font=\scriptsize]  {$( x,0)$};
\draw (203,144.5) node [anchor=north west][inner sep=0.75pt]  [font=\scriptsize]  {$( y,0)$};
\draw (322,97.5) node [anchor=north west][inner sep=0.75pt]  [font=\scriptsize]  {$( x,1)$};
\draw (202,97) node [anchor=north west][inner sep=0.75pt]  [font=\scriptsize]  {$( y,1)$};

\end{tikzpicture}
    \end{center}

  We can reverse the construction as follows. Given an upward-closed set $I\subseteq \chi (R)$ such that $I\cup \chi(Q)=\chi(R)$, let $\widetilde{V}$ denote the subset of $P$ consisting of those elements whose connected component $C$ in $R$ satisfies that $C \cap (P \setminus I)$ is non-empty. Now construct the relation $S$ on $P\times [n]$ as the closure under transitivity of 
  \[T=(Q \times \{0\}) \cup (R \times \{1\}) \cup (\widetilde{V} \times [1]).\]
  We claim that $S$ is a saturated transfer system on $P \times [1]$ with top and bottom layers given by $R$ and $Q$, respectively. The latter follows directly from the construction, since closing under transitivity will only add arrows that connect the bottom layer with the top layer. 
  
  To prove $S$ is closed under restriction it is enough to prove that $T$ itself is as well, as this property is respected by closing under transitivity (see \cite[Theorem A.2]{Rubin}). We split into cases. Since $Q$ is itself closed under restriction, the restriction of an arrow in $Q \times \{0\}$ is in $Q\times \{0\}$. The same is true for arrows in $R \times \{1\}$ when restricting along elements of the form $(y,1)$. The restriction of $(x,1) \to (z,1)$ along $(y,0)$ is the arrow $(x\wedge y, 0) \to (z\wedge y, 0)$. Now, if the former is in $R\times \{1\}$, we have that $(x\wedge y, 1) \to (z\wedge y, 1)$ is also in $R\times \{1\}$ because $R$ is closed under restriction, and from there we get that $(x\wedge y, 0) \to (z\wedge y, 0)$ is in $Q \times \{0\}$, since $R \leq Q$. 

  Finally, we consider arrows in $\widetilde{V} \times [1]$. Showing that this set is closed under restriction is equivalent to showing that $\widetilde{V}$ is downward-closed. Let $x\in \widetilde{V}$ and $y\leq x$. Then there exists $z\in P\setminus I$ such that $x$ and $z$ are in the same connected component with respect to $R$, meaning that there is a zig-zag of arrows in $R$ connecting $x$ and $z$. Taking the restriction of this zig-zag with respect to $y$, we obtain a zig-zag of arrows in $R$ connecting $y=x\wedge y$ with $z\wedge y$. Now, since $I$ is upward-closed we have that $P\setminus I$ is downward-closed, and hence, $z\wedge y$ is in $P\setminus I$. This implies that $y\in \widetilde{V}$, as wanted.

  It remains to show that $S$ is decomposable. Note that $T$ is decomposable, since $R$ and $Q$ are as well, and the arrows in $\widetilde{V}\times [1]$ are not composites. The only new arrows that are in $S$ but not in $T$ are of the form $(x,0)\to (y,1)$ with $x<y$. Since $R$ and $Q$ are closed under transitivity, if such an arrow is in $S$, it is because there exists  $z\in \widetilde{V}$, such that $xQz$ and $zRy$, as below.
  \begin{center}
      \tikzset{every picture/.style={line width=0.75pt}} 

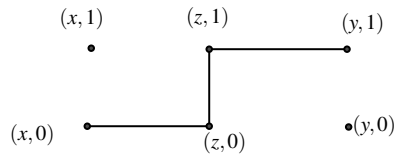
\begin{tikzpicture}[x=0.75pt,y=0.75pt,yscale=-1,xscale=1]

\draw    (218.5,239.5) -- (280,239.5) ;
\draw    (280,200.62) -- (280,239.5) ;
\draw    (280,200.62) -- (349.5,200.62) ;
\draw  [fill={rgb, 255:red, 74; green, 74; blue, 74 }  ,fill opacity=1 ] (348.94,239.88) .. controls (348.94,240.64) and (349.56,241.26) .. (350.32,241.26) .. controls (351.08,241.26) and (351.7,240.64) .. (351.7,239.88) .. controls (351.7,239.12) and (351.08,238.5) .. (350.32,238.5) .. controls (349.56,238.5) and (348.94,239.12) .. (348.94,239.88) -- cycle ;
\draw  [fill={rgb, 255:red, 74; green, 74; blue, 74 }  ,fill opacity=1 ] (348.12,200.62) .. controls (348.12,201.38) and (348.74,202) .. (349.5,202) .. controls (350.26,202) and (350.88,201.38) .. (350.88,200.62) .. controls (350.88,199.86) and (350.26,199.24) .. (349.5,199.24) .. controls (348.74,199.24) and (348.12,199.86) .. (348.12,200.62) -- cycle ;
\draw  [fill={rgb, 255:red, 74; green, 74; blue, 74 }  ,fill opacity=1 ] (278.62,239.5) .. controls (278.62,240.26) and (279.24,240.88) .. (280,240.88) .. controls (280.76,240.88) and (281.38,240.26) .. (281.38,239.5) .. controls (281.38,238.74) and (280.76,238.12) .. (280,238.12) .. controls (279.24,238.12) and (278.62,238.74) .. (278.62,239.5) -- cycle ;
\draw  [fill={rgb, 255:red, 74; green, 74; blue, 74 }  ,fill opacity=1 ] (278.62,200.62) .. controls (278.62,201.38) and (279.24,202) .. (280,202) .. controls (280.76,202) and (281.38,201.38) .. (281.38,200.62) .. controls (281.38,199.86) and (280.76,199.24) .. (280,199.24) .. controls (279.24,199.24) and (278.62,199.86) .. (278.62,200.62) -- cycle ;
\draw  [fill={rgb, 255:red, 74; green, 74; blue, 74 }  ,fill opacity=1 ] (219.12,200.12) .. controls (219.12,200.88) and (219.74,201.5) .. (220.5,201.5) .. controls (221.26,201.5) and (221.88,200.88) .. (221.88,200.12) .. controls (221.88,199.36) and (221.26,198.74) .. (220.5,198.74) .. controls (219.74,198.74) and (219.12,199.36) .. (219.12,200.12) -- cycle ;
\draw  [fill={rgb, 255:red, 74; green, 74; blue, 74 }  ,fill opacity=1 ] (217.12,239.5) .. controls (217.12,240.26) and (217.74,240.88) .. (218.5,240.88) .. controls (219.26,240.88) and (219.88,240.26) .. (219.88,239.5) .. controls (219.88,238.74) and (219.26,238.12) .. (218.5,238.12) .. controls (217.74,238.12) and (217.12,238.74) .. (217.12,239.5) -- cycle ;

\draw (345,181) node [anchor=north west][inner sep=0.75pt]  [font=\scriptsize]  {$( y,1)$};
\draw (350.82,232.5) node [anchor=north west][inner sep=0.75pt]  [font=\scriptsize]  {$( y,0)$};
\draw (178,237) node [anchor=north west][inner sep=0.75pt]  [font=\scriptsize]  {$( x,0)$};
\draw (203,178.5) node [anchor=north west][inner sep=0.75pt]  [font=\scriptsize]  {$( x,1)$};
\draw (275.32,241) node [anchor=north west][inner sep=0.75pt]  [font=\scriptsize]  {$( z,0)$};
\draw (266.82,178) node [anchor=north west][inner sep=0.75pt]  [font=\scriptsize]  {$( z,1)$};

\end{tikzpicture}
  \end{center}
  \noindent Thus $y$ and $z$ are in the same connected component with respect to $R$, and hence $y \in \widetilde{V}$.  
  
  Let $w$ be the minimal element in the connected component of $z$ with respect to $R$, as depicted by the top red arrow in the diagram below.
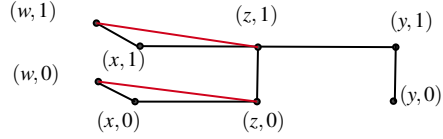
\begin{center}
\tikzset{every picture/.style={line width=0.75pt}} 

\begin{tikzpicture}[x=0.75pt,y=0.75pt,yscale=-1,xscale=1]

\draw    (100,370.12) -- (161.5,370.12) ;
\draw    (162,342.62) -- (161.5,371.5) ;
\draw    (103.88,342.12) -- (231.5,342.62) ;
\draw  [fill={rgb, 255:red, 74; green, 74; blue, 74 }  ,fill opacity=1 ] (228.94,369.88) .. controls (228.94,370.64) and (229.56,371.26) .. (230.32,371.26) .. controls (231.08,371.26) and (231.7,370.64) .. (231.7,369.88) .. controls (231.7,369.12) and (231.08,368.5) .. (230.32,368.5) .. controls (229.56,368.5) and (228.94,369.12) .. (228.94,369.88) -- cycle ;
\draw  [fill={rgb, 255:red, 74; green, 74; blue, 74 }  ,fill opacity=1 ] (230.12,342.62) .. controls (230.12,343.38) and (230.74,344) .. (231.5,344) .. controls (232.26,344) and (232.88,343.38) .. (232.88,342.62) .. controls (232.88,341.86) and (232.26,341.24) .. (231.5,341.24) .. controls (230.74,341.24) and (230.12,341.86) .. (230.12,342.62) -- cycle ;
\draw  [fill={rgb, 255:red, 74; green, 74; blue, 74 }  ,fill opacity=1 ] (160.12,370.12) .. controls (160.12,370.88) and (160.74,371.5) .. (161.5,371.5) .. controls (162.26,371.5) and (162.88,370.88) .. (162.88,370.12) .. controls (162.88,369.36) and (162.26,368.74) .. (161.5,368.74) .. controls (160.74,368.74) and (160.12,369.36) .. (160.12,370.12) -- cycle ;
\draw  [fill={rgb, 255:red, 74; green, 74; blue, 74 }  ,fill opacity=1 ] (160.62,342.62) .. controls (160.62,343.38) and (161.24,344) .. (162,344) .. controls (162.76,344) and (163.38,343.38) .. (163.38,342.62) .. controls (163.38,341.86) and (162.76,341.24) .. (162,341.24) .. controls (161.24,341.24) and (160.62,341.86) .. (160.62,342.62) -- cycle ;
\draw  [fill={rgb, 255:red, 74; green, 74; blue, 74 }  ,fill opacity=1 ] (101.12,342.12) .. controls (101.12,342.88) and (101.74,343.5) .. (102.5,343.5) .. controls (103.26,343.5) and (103.88,342.88) .. (103.88,342.12) .. controls (103.88,341.36) and (103.26,340.74) .. (102.5,340.74) .. controls (101.74,340.74) and (101.12,341.36) .. (101.12,342.12) -- cycle ;
\draw  [fill={rgb, 255:red, 74; green, 74; blue, 74 }  ,fill opacity=1 ] (98.62,370.12) .. controls (98.62,370.88) and (99.24,371.5) .. (100,371.5) .. controls (100.76,371.5) and (101.38,370.88) .. (101.38,370.12) .. controls (101.38,369.36) and (100.76,368.74) .. (100,368.74) .. controls (99.24,368.74) and (98.62,369.36) .. (98.62,370.12) -- cycle ;
\draw    (231.5,341.24) -- (231,370.12) ;
\draw    (100,370.12) -- (80.09,360.02) ;
\draw    (101.12,342.12) -- (80.52,330.49) ;
\draw  [fill={rgb, 255:red, 74; green, 74; blue, 74 }  ,fill opacity=1 ] (79.14,330.49) .. controls (79.14,331.25) and (79.76,331.87) .. (80.52,331.87) .. controls (81.28,331.87) and (81.9,331.25) .. (81.9,330.49) .. controls (81.9,329.72) and (81.28,329.1) .. (80.52,329.1) .. controls (79.76,329.1) and (79.14,329.72) .. (79.14,330.49) -- cycle ;
\draw  [fill={rgb, 255:red, 74; green, 74; blue, 74 }  ,fill opacity=1 ] (80.09,360.02) .. controls (80.09,360.78) and (80.71,361.4) .. (81.47,361.4) .. controls (82.24,361.4) and (82.86,360.78) .. (82.86,360.02) .. controls (82.86,359.25) and (82.24,358.63) .. (81.47,358.63) .. controls (80.71,358.63) and (80.09,359.25) .. (80.09,360.02) -- cycle ;
\draw [color={rgb, 255:red, 208; green, 2; blue, 27 }  ,draw opacity=1 ]   (80.52,330.49) -- (162,342.62) ;
\draw [color={rgb, 255:red, 208; green, 2; blue, 27 }  ,draw opacity=1 ]   (81.47,360.02) -- (161.5,370.12) ;

\draw (227,323) node [anchor=north west][inner sep=0.75pt]  [font=\scriptsize]  {$( y,1)$};
\draw (231.82,360.5) node [anchor=north west][inner sep=0.75pt]  [font=\scriptsize]  {$( y,0)$};
\draw (79,373.5) node [anchor=north west][inner sep=0.75pt]  [font=\scriptsize]  {$( x,0)$};
\draw (81.52,342.6) node [anchor=north west][inner sep=0.75pt]  [font=\scriptsize]  {$( x,1)$};
\draw (152.32,373.5) node [anchor=north west][inner sep=0.75pt]  [font=\scriptsize]  {$( z,0)$};
\draw (148.82,320) node [anchor=north west][inner sep=0.75pt]  [font=\scriptsize]  {$( z,1)$};
\draw (35,318) node [anchor=north west][inner sep=0.75pt]  [font=\scriptsize]  {$( w,1)$};
\draw (36.5,349.5) node [anchor=north west][inner sep=0.75pt]  [font=\scriptsize]  {$( w,0)$};
\end{tikzpicture}
\end{center}
  \noindent Note that $w\in P\setminus I$, since $P\setminus I$ is downward-closed and the intersection of the connected component of $z$ with $P \setminus I$ is non-empty. Given that $w\in \chi (R)=I \cup \chi(Q)$, we then must have that $w\in \chi(Q)$. Thus $w$ is the minimal element in its connected component with respect to $Q$, and since $R\leq Q$, we have the red arrow at the bottom, implying that this is the connected component of $z$. Given that $xQ z$, this is also the connected component of $x$, and thus we have $w\leq x \leq z$. Since $w R z$, decomposability of $R$ implies that $xRz$, and hence $xRy$.

  Any decomposition of $(x,0)\to(y,1)$ is of the form $(x,0)\to(u,0)\to(y,1)$ or $(x,0)\to(u,1)\to(y,1)$ for some $u$ such that $x\leq u \leq y$. Since $xRy$ and $R$ is decomposable, we have $xRu$ and $uRy$, meaning that $(x,0) \to (u,0)$ and $(u,1) \to (y,1)$ are both in $S$. Further, since $\widetilde{V}$ is downward-closed, we also have that $(u,0)\to (u,1)$ is in $S$. Closure under transitivity implies then that both of the components of either our decompositions are in $S$, proving decomposability.
  
  We leave it to the reader to prove that these two constructions are inverses of each other. 
\end{proof}

\appendix
\section{Explicit computations}\label{sec:table}

The following table records the output of a short Sage Math script that, for $P$ a finite join-semilattice, computes $W(P)$, its diagonalization, and an eigenbasis. This allows us to determine the eigenvalues $\Lambda=\Lambda(P)$ and coefficients $b_\lambda\in \QQ$, $\lambda\in \Lambda$ of \autoref{cor:fmla}, giving
\[
  \#\SubMon(P\times\chain n) = \sum_{\lambda\in \Lambda}b_\lambda\lambda^n.
\]
The table displays the finite join-semilattice $P$, eigenvalues $\lambda\in \Lambda(P)$, coefficients $b_\lambda\in \QQ$, and normalized coefficients $\hat b_\lambda := b_\lambda\prod_{\mu\ne \lambda\in \Lambda}(\lambda-\mu)\in \ZZ$ (see \autoref{rmk:interpretation-of-coeffs}).

We note that our computations reveal that $4\in \Lambda(M_3)$ but $b_4(M_3)=0$. This is the only example we are aware of for which $b_\lambda=0$.

{\renewcommand{\arraystretch}{1.2}%
\begin{longtable}{@{}llll@{}}
\caption{Several finite join-semilattices $P$, and the associated eigenvalues and normalized coefficients $\hat b_\lambda := b_\lambda\prod_{\mu\neq\lambda} (\lambda - \mu)$, as discussed in \aref{rmk:interpretation-of-coeffs}. The cases included are either products of chains, lattices $M_k = \{\bot,1,2,\ldots,k,\top\}$ with integer elements incomparable, or the pentagon lattice $N_5$.}
\label{table:normalized-coeffs}\\
\toprule
P & $\lambda$ & $b_\lambda$ & $\hat b_\lambda$\\
\midrule
\endfirsthead

\toprule
P & $\lambda$ & $b_\lambda$ & $\hat b_\lambda$\\
\midrule
\endhead

\midrule
\multicolumn{4}{r}{\emph{Continued on next page}}\\
\endfoot

\bottomrule
\endlastfoot

$[1]\times [1]$ & 2 & 1/2 & 4\\
 & 3 & 1 & -3\\
 & 4 & -12 & -48\\
 & 6 & 35/2 & -420\\
\addlinespace

$[2]\times [1]$ & 2 & -13/48 & -1560\\
 & 3 & -111/35 & 2664\\
 & 4 & 14 & 4032\\
 & 5 & 16 & -2880\\
 & 6 & 315/8 & 7560\\
 & 7 & -154 & 55440\\
 & 8 & -406/15 & -38976\\
 & 10 & 15471/112 & -5569560\\
\addlinespace

$[3]\times [1]$ & 2 & 7213/34320 & 109060560\\
 & 3 & 28631/8400 & -148423104\\
 & 4 & -101729/13860 & -58595904\\
 & 5 & -797/15 & 128540160\\
 & 6 & 719/24 & 32613840\\
 & 7 & -45157/120 & 260104320\\
 & 8 & 59599/75 & 480606336\\
 & 9 & 1482/5 & -215115264\\
 & 10 & 735867/560 & 1589472720\\
 & 11 & -840037/336 & 7257919680\\
 & 12 & -42493/60 & -7709929920\\
 & 13 & -525746/2475 & 16958462976\\
 & 15 & 12800093/8580 & 9289795495680\\
\addlinespace

$[4]\times [1]$ & 2 & -335766173/1551950400 & -1462113946298880\\
 & 3 & -5004577/1501500 & 1255261295296512\\
 & 4 & 32846000389/3216213000 & 454063109377536\\
 & 5 & 9458167/96096 & -823723113830400\\
 & 6 & -423949649/960960 & -988989741187200\\
 & 7 & 799421317/831600 & -773584020034560\\
 & 8 & 250240977/143000 & 653813609283072\\
 & 9 & -17901851/5250 & 748364969189376\\
 & 10 & 22150545/4928 & 723419079264000\\
 & 11 & -796666013/23520 & 4955899933670400\\
 & 12 & 2096204/75 & 4600535229480960\\
 & 13 & 45969602783/4158000 & -2541935155488768\\
 & 14 & 15454724597/1848000 & 3364926277055616\\
 & 15 & 21650713/572 & -33940558368460800\\
 & 16 & -26792689523/540540 & -129633748988083200\\
 & 17 & -1273644454/75075 & 177477158195527680\\
 & 18 & -147245784937/21021000 & -439673557889323008\\
 & 19 & -168120903799/68068000 & 1757022150832736256\\
 & 21 & 665038415449/31039008 & 2606354714139837696000\\
\addlinespace

$[1]^3$ & 2 & 2453/12480 & 2649240\\
 & 3 & -8621/21420 & 620712\\
 & 4 & -351/32 & -4447872\\
 & 5 & 108/175 & -116640\\
 & 6 & -13385/288 & -6746040\\
 & 7 & 96831/520 & -34859160\\
 & 8 & 493/10 & 23853312\\
 & 10 & -2956707/11200 & 1064414520\\
 & 12 & -31273/2520 & -360264960\\
 & 15 & -3336709/23400 & 92493573480\\
 & 20 & 159923969/530400 & 19344403290240\\
\addlinespace

$M_3$ & 2 & 3/16 & 12\\
 & 3 & 3/7 & -9\\
 & 4 & 0 & 0\\
 & 6 & -105/8 & 1260\\
 & 10 & 2745/112 & 32940\\
\addlinespace

$M_4$ & 2 & 1/64 & 16\\
 & 3 & 3/35 & -27\\
 & 4 & 12/7 & 576\\
 & 6 & -35/8 & 5040\\
 & 10 & -2745/224 & -131760\\
 & 18 & 80223/2240 & -11552112\\
\addlinespace

$M_5$ & 2 & -125/2048 & -2000\\
 & 3 & -19/217 & 855\\
 & 4 & 12/7 & 17280\\
 & 6 & -25/32 & 25200\\
 & 10 & -13725/1792 & -1976400\\
 & 18 & -80223/7168 & 57760560\\
 & 34 & 3559545/63488 & 17940106800\\
\addlinespace

$M_6$ & 2 & -5241/65536 & -167712\\
 & 3 & -225/1519 & 91125\\
 & 4 & 36/31 & 725760\\
 & 6 & 75/64 & -2268000\\
 & 10 & -233325/50176 & -67197600\\
 & 18 & -80223/8192 & 2425943520\\
 & 34 & -10678635/1015808 & -107640640800\\
 & 66 & 9342726965/99549184 & -121081741466400\\
\addlinespace

$M_7$ & 2 & -574609/8388608 & -18387488\\
 & 3 & -3933/27559 & 11150055\\
 & 4 & 110/217 & 39916800\\
 & 6 & 16205/7936 & -490039200\\
 & 10 & -112545/57344 & -3403360800\\
 & 18 & -7781631/917504 & 235316521440\\
 & 34 & -373752225/32505856 & -11302267284000\\
 & 66 & -9342726965/910163968 & 847572190264800\\
 & 130 & 38345903755445/231181647872 & 3478740388693970400\\
\addlinespace

$N_5$ & 2 & -1/3 & -48\\
 & 3 & 11/5 & -66\\
 & 4 & 12 & 192\\
 & 5 & -20 & 360\\
 & 6 & -35 & -1680\\
 & 8 & 812/15 & -38976\\

\end{longtable}%
}

\bibliographystyle{alpha}
\bibliography{closure}

\newcommand{\etalchar}[1]{$^{#1}$}
\begin{thebibliography}{HMOO22}

\bibitem[BCH{\etalchar{+}}25]{mrc-ormsby}
J.~Bose, T.~Chih, H.~Housden, L.~Jones II, C.~Lewis, K.~Ormsby, and M.~Rose.
\newblock Combinatorics of factorization systems on lattices, 2025.

\bibitem[BH15]{BlumbergHill}
A.J. Blumberg and M.A. Hill.
\newblock Operadic multiplications in equivariant spectra, norms, and
  transfers.
\newblock {\em Adv. Math.}, 285:658--708, 2015.

\bibitem[BHK{\etalchar{+}}25]{echt-2023}
L.~Bao, C.~Hazel, T.~Karkos, A.~Kessler, A.~Nicolas, K.~Ormsby, J.~Park,
  C.~Schleff, and S.~Tilton.
\newblock Transfer systems for rank two elementary abelian groups:
  characteristic functions and matchstick games.
\newblock {\em Tunis. J. Math.}, 7(1):167--191, 2025.

\bibitem[BHO{\etalchar{+}}24]{notices_transfer}
A.J. Blumberg, M.A. Hill, K.~Ormsby, A.M. Osorno, and C.~Roitzheim.
\newblock Homotopical combinatorics.
\newblock {\em Notices Amer. Math. Soc.}, 71(2):260--266, 2024.

\bibitem[HMOO22]{hmoo}
U.~Hafeez, P.~Marcus, K.~Ormsby, and A.M. Osorno.
\newblock Saturated and linear isometric transfer systems for cyclic groups of
  order {$p^mq^n$}.
\newblock {\em Topology Appl.}, 317:Paper No. 108162, 2022.

\bibitem[hu]{mse-diagonalizability}
hardmath (\url{https://math.stackexchange.com/users/3111/hardmath}).
\newblock If a matrix is triangular, is there a quicker way to tell if it is
  can be diagonalized?
\newblock Mathematics Stack Exchange.
\newblock URL:\url{https://math.stackexchange.com/q/1558609} (version:
  2017-04-13).

\bibitem[Kan97]{kaneko}
M.~Kaneko.
\newblock Poly-{B}ernoulli numbers.
\newblock {\em J. Th\'eor. Nombres Bordeaux}, 9(1):221--228, 1997.

\bibitem[Knu24]{knuth:parades}
D.~Knuth.
\newblock Parades and poly-{B}ernoulli numbers, 2024.
\newblock URL:
  \url{https://www-cs-faculty.stanford.edu/~knuth/papers/poly-Bernoulli.pdf}
  accessed August 2025.

\bibitem[LeV07]{leveque2007}
R.J. LeVeque.
\newblock {\em Finite Difference Methods for Ordinary and Partial Differential
  Equations}.
\newblock Society for Industrial and Applied Mathematics, 2007.

\bibitem[OLBC10]{nist-handbook}
F.~Olver, D.~Lozier, R.~Boisvert, and C.~Clark.
\newblock {\em The NIST Handbook of Mathematical Functions}.
\newblock Cambridge University Press, 2010.

\bibitem[Rub21]{Rubin}
J.~Rubin.
\newblock Detecting {S}teiner and linear isometries operads.
\newblock {\em Glasg. Math. J.}, 63(2):307--342, 2021.

\bibitem[Sch94]{schmidt}
R.~Schmidt.
\newblock {\em Subgroup lattices of groups}, volume~14 of {\em De Gruyter
  Expositions in Mathematics}.
\newblock Walter de Gruyter \& Co., Berlin, 1994.

\bibitem[Sta11]{EC}
Richard~P. Stanley.
\newblock {\em Enumerative Combinatorics: Volume 1}.
\newblock Cambridge University Press, USA, 2nd edition, 2011.

\end{thebibliography}

\end{document}